\documentclass[12pt,reqno]{amsart}

\usepackage{amssymb}
\usepackage{amsmath}
\usepackage{hyperref}
\usepackage[all]{xypic}
\usepackage{color}
\usepackage{tikz}
\usetikzlibrary{matrix,arrows}

\newcommand{\cor}{\text{cor}}
\newcommand{\F}{\mathbb{F}}

\newcommand{\Gal}{\text{\rm Gal}}

\newcommand{\N}{\mathbb{N}}

\newcommand{\Z}{\mathbb{Z}}

\newcommand{\ch}[1]{\textrm{char(}#1\textrm{)}}
\newcommand{\comment}[1]{}

\begin{document}

\title[Galois $p$-groups and Galois modules]{Galois $p$-groups and Galois modules}

\author[S.~Chebolu]{Sunil Chebolu}
\address{Department of Mathematics, Illinois State University, Campus Box 4520, Normal, IL \  61790 \ USA}
\email{schebol@ilstu.edu}
\thanks{The first author is partially supported by NSA grant H98230-13-1-0238}

\author[J.~Min\'{a}\v{c}]{J\'an Min\'a\v{c}}
\address{Department of Mathematics, Middlesex College, \ University
of Western Ontario, London, Ontario \ N6A 5B7 \ CANADA}
\email{minac@uwo.ca}
\thanks{The second author is partially supported by NSERC grant R0370A01}

\author[A.~Schultz]{Andrew Schultz}
\address{Department of Mathematics, Wellesley College, 106 Central Street, Wellesley, MA \ 01702 \ USA}
\email{andrew.c.schultz@gmail.com}



\begin{abstract}
The smallest non-abelian $p$-groups play a fundamental role in the theory of Galois $p$-extensions.  We illustrate this by highlighting their role in the definition of the norm residue map in Galois cohomology.  We then determine how often these groups --- as well as other closely related, larger $p$-groups --- occur as Galois groups over given base fields.  We show further how the appearance of some Galois groups forces the appearance of other Galois groups in an interesting way.
\end{abstract}

\dedicatory{Dedicated to Professor Albrecht Pfister.}
\date{\today}

\maketitle

\newtheorem*{theorem*}{Theorem}
\newtheorem*{lemma*}{Lemma}
\newtheorem{proposition}{Proposition}[section]
\newtheorem{theorem}[proposition]{Theorem}
\newtheorem{corollary}[proposition]{Corollary}
\newtheorem{lemma}[proposition]{Lemma}
\newtheorem{conjecture}[proposition]{Conjecture}

\theoremstyle{definition}
\newtheorem*{definition*}{Definition}
\newtheorem*{remark*}{Remark}
\newtheorem{definition}[proposition]{Definition}
\newtheorem{example}[proposition]{Example}

\parskip=10pt plus 2pt minus 2pt

\section{Introduction}

From the very beginning Galois theory carries an aura of mystery and depth.  Despite some remarkable progress, some very basic problems remain open.  Given a base field $F$ and a finite group $G$, the inverse Galois problem asks whether there is a Galois extension $K/F$ with Galois group $G$.  Even when $F$ is the field of rational numbers we do not know an answer for all $G$.  However in the remarkable paper \cite{Sha1}, I.R.~Shafarevich showed that each solvable group $G$ appears as a Galois group over any algebraic number field.  (See also \cite{Sha2} for corrections related to problems caused by the prime $2$.)  For the comprehensive treatment of Shafarevich's theorem over any global field, see \cite[Sec.~9.6]{NSW}.  For another nice exposition of Shafarevich's theorem in the original case of algebraic number fields, see  \cite[Ch.V, \S 6]{ILF}. For the special case of solvable groups $G$ whose order is a power of a prime $p$, see \cite{Sha3}.  

Although in the works above there are some considerable technical issues, some basic principles can be explained briefly.  The first basic ingredient is the description of all Galois extensions $K/F$ with Galois group $\Z/p$ over any field $F$.  If $F$ contains a primitive $p$th root of unity, this problem has a classical, elegant solution described by Kummer theory; if $\ch{F} = p$ the problem is solved by Artin-Schreier theory, and when $\ch{F}\neq p$ and $F$ does not contain  primitive $p$th root of unity, one can use Galois descent (see \cite[Ch.~6]{Lang}).  The appearance of elementary $p$-abelian groups can be described in this language as well.  Indeed, for any field $F$ there exists an $\F_p$-space $J(F)$ so that subspaces of $J(F)$ of dimension $k$ are in correspondence with elementary $p$-abelian extensions $L/F$ of rank $k$.  (We will consider $J(F)$ more fully in section \ref{sec:recasting}.)

In order to build a given group $G$ of prime-power order as a Galois group over a given field $F$, it is natural to consider Galois embedding problems.  Consider the short exact sequence 
$$\xymatrix{1 \ar[r]& \Gal(L/K) \ar[r]&\Gal(L/F) \ar[r]& \Gal(K/F) \ar[r]&1}$$
from Galois theory. If you are given $G \twoheadrightarrow Q$ and you know $\Gal(K/F) \simeq Q$, can you find a Galois extension $L/F$ as above so that $\Gal(L/F) \simeq G$, with the natural map from Galois theory corresponding to the original surjection $G \twoheadrightarrow Q$?  This is the Galois embedding problem for $G \twoheadrightarrow Q$ over the extension $K/F$.

Consider first $p>2$. The smallest nonabelian groups of prime-power order have order $p^3$, and up to group isomorphism there are exactly two such groups (see, e.g., \cite[p.~185--186]{DF}).  The Heisenberg group, which we write $H_{p^3}$, is the unique nonabelian group of order $p^3$ and exponent $p$.  This group is isomorphic to $U_3(\F_p)$, the group of upper triangular matrices over a field with $p$-elements with all diagonal elements $1$.  The modular group, which we write $M_{p^3}$, is the unique nonabelian group of order $p^3$ and exponent $p^2$.  If $p=2$ there are also two nonabelian groups of order $8$: the dihedral group $D_4$ (isomorphic to $U_3(\F_2$)) and the quaternion group $Q_8$.  $H_{p^3}$ and $M_{p^3}$ have the following presentations by generators and relations:
\begin{equation}\label{eq:group.presentations}
\begin{split}
H_{p^3} &= \left\langle \sigma,\tau,\omega~|~\sigma^p=\tau^p=\omega^p = [\omega,\sigma] = [\omega,\tau] = \mbox{id}, [\sigma,\tau] = \omega \right\rangle\\
M_{p^3} &=\left\langle x,y~|~y^{p^2}=x^p=\mbox{id}, [x,y]=y^p \right\rangle = \langle y \rangle \rtimes \langle x \rangle.
\end{split}
\end{equation}

It is worth observing that $M_{p^3}$ is one group in the larger family of groups of the form $M_{p^n} = \langle y: y^{p^{n-1}}=1\rangle \rtimes \langle x: x^p=1\rangle$.  When $p=2$ and $n>3$, there are four non-abelian groups of order $2^n$ which have an element of order $2^{n-1}$; $M_{2^n}$ is obviously one of them, and some automatic realization results for these four groups were considered in \cite{J1}.  (Jensen uses $F_{2^n}$ to denote the group we're calling $M_{2^n}$.)  For odd $p$, the group $M_{p^n}$ is the only non-abelian group of order $p^{n}$ which contains an element of order $p^{n-1}$ (see, e.g., \cite[\S 12.5]{Hall}). Michailov described in \cite{Mich1} the realizability conditions and Galois extensions for $M_{p^n}$, as well as several other groups related to this group.

The nonabelian groups of order $p^3$, along with the cyclic groups of order dividing $p^2$, play a surprising basic role in the structure of some canonical quotients of absolute Galois groups of fields containing a primitive $p$th root of unity.  For $p=2$, based on work of F.~Villegas in \cite{MSp}, it was shown that the fixed field of the third term of the $2$-descending central sequence of the absolute Galois group is the compositum of all Galois extensions of the base field with Galois group isomorphic to $\Z/2$, $\Z/4$ or the dihedral group of order $8$.  Similar results were obtained in \cite{EM11} for $p>2$ with groups $\Z/p, \Z/p^2$ and $M_{p^3}$.  When replacing the $2$-descending central sequence by the Zassenhaus descending sequence in \cite{EM2}, an analogous result for groups $\Z/p$ and $H_{p^3}$ was shown.  On the other hand, in \cite{CEM,EM2,EM12,MSp2} it was shown that these quotients of absolute Galois groups of fields as above encode crucial information about Galois cohomology, valuations and orderings of fields.  These fundamental results can be viewed also as precursors to current investigations in \cite{Ef,Ef2,HW,MT4,MT1,MT2,MT3} related to Massey products in Galois cohomology.

If $F$ is a field containing a primitive $p$th root of unity $\xi_p$, and if $a,b \in F^\times$ give rise to a $\Z/p \times \Z/p$ extension $F(\root{p}\of{a},\root{p}\of{b})/F$, then it is known (see \cite[(3.6.4)]{Lbook}, \cite[Th.~3(A)]{Mass},\cite[Th.~3.1]{Mich2}) that the embedding problem corresponding to
\begin{equation}\label{eq:hp3.embedding.problem}\xymatrix{1 \ar[r] & \Z/p \ar[r]&H_{p^3} \ar[r] & \Z/p \times \Z/p \ar[r] & 1}\end{equation}
is solvable over $F(\root{p}\of{a},\root{p}\of{b})$ if and only if $b \in N_{F(\root{p}\of{a})/F}(F(\root{p}\of{a})^\times)$.  The analogous embedding problem for $M_{p^3}$ is solvable over $F(\root{p}\of{a},\root{p}\of{b})$ if and only if $b \xi_p^{k} \in N_{F(\root{p}\of{a})/F}(F(\root{p}\of{a})^\times)$ for some $k \in \Z \setminus p\Z$  (see \cite[Cor, p.~523--524]{Mass}, \cite[Th.~1]{Br} or \cite[Th.~3.2]{Mich2}.)  It is also known that the existence of an extension $K/F$ with $\Gal(K/F) \simeq H_{p^3}$ implies the existence of an extension $L/F$ with $\Gal(L/F) \simeq M_{p^3}$ (see \cite[Thm.~2]{Br}). In some special cases, more precise results were obtained.  Let $\nu(F,G)$ be the number of extensions of $F$ whose Galois group is $G$ (in a fixed algebraic closure $\bar F$), then Brattstr\"om has proved (\cite[Th.~5]{Br}) that when $\ch{F}=p$ or $\xi_{p^2} \in F$, one has
$$\nu(F,M_{p^3}) = (p^2-1)\nu(F,H_{p^3}).$$  

The purpose of this paper is to investigate $H_{p^3}, M_{p^3}$ and their closely related $p$-groups as Galois groups.  We investigate the number of Galois extensions with given Galois group $G$ as above and the interrelation between these numbers.  Section \ref{sec:norm.residue} discusses the fundamental relation in Galois cohomology.  Assume that $a$ and $1-a$ are non-zero element in a field $F$ and $\xi_p \in F$.  Then by Kummer theory we have associated classes $(a), (1-a) \in H^1(G_F,\F_p)$.  Here $G_F$ is the absolute Galois group of $F$ and $H^i(G_F,\F_p)$ is the $i$th group cohomology of $G_F$ with $\F_p$-coefficients viewed as a trivial module over $G_F$.  Then $(a) \cup (1-a) =0 \in H^2(G_F,\F_p)$, a result known as the Bass-Tate lemma (see Proposition \ref{prop:basstate} below).  We discuss how this relation is connected to the solvability of embedding problem (\ref{eq:hp3.embedding.problem}) over the field $F(\root{p}\of{a},\root{p}\of{1-a})$.  We will use this as motivation for providing an ``elementary" proof of this vanishing for $p>2$ akin to the result of Pfister for $p=2$ in \cite{Pf}.  Pfisters's proof is remarkable as it uses nothing but the definition of Galois cohomology, but the impetus behind the selection of the desired coboundary is not explained.  In our treatment we also keep the elementary nature of the proof, but we shed light on the construction of the desired coboundary.  

The Bass-Tate lemma has a fairly short proof (see section \ref{sec:norm.residue} below), but it is nevertheless a deep statement which implies automatic realizations for Galois groups $H_{p^3}$, $\Z/4$ and the dihedral group of order $8$.  It also connects Galois theory directly to the crucial relation between addition and multiplication in the base field.  This connection has some profound consequences for birational anabelian geometry (see \cite{BT1,BT2,BT3,Pop1,Pop2,T1,T2}).

In section \ref{sec:recasting} we change our methodology for investigating $H_{p^3}$ and $M_{p^3}$ extensions by replacing the embedding problem from (\ref{eq:hp3.embedding.problem}) to one that comes from a short exact sequence of the form \begin{equation}\label{eq:hp3.embedding.problem.redux}\xymatrix{1 \ar[r] & \Z/p \times \Z/p \ar[r]&G \ar[r] & \Z/p \ar[r] & 1.}\end{equation}  (Both $H_{p^3}$ and $M_{p^3}$ fit in such an exact sequence.) If $K/F$ is the base extension satisfying $\Gal(K/F) \simeq \Z/p$, this new perspective allows us to parameterize solutions to these embedding problems over $K/F$ by certain $\Gal(K/F)$-submodules within $J(K)$.  This change in perspective also allows us to exhibit $H_{p^3}$ and $M_{p^3}$ within a larger family of $p$-groups (namely those which can be written as an extension of $\Z/p^n$ by a cyclic $\F_p[\Z/p^n]$-module) for which the solvability of the associated embedding problem is again tied to the appearance of certain Galois modules in $J(K)$.   In the final two sections we use this module-theoretic perspective to give generalizations of some of the known results concerning the appearance of $M_{p^3}$ and $H_{p^3}$ as Galois groups to this broader family of groups.

\subsection{Acknowledgements} We gratefully acknowledge discussions and collaborations with our friends and colleagues I.~Efrat, J.~G\"{a}rtner, S.~Gille, D.~Hoffmann, J.~Labute, J.~Swallow, N.D.~Tan, A.~Topaz, R.~Vakil and K.~Wickelgren which influenced our thinking and the results presented in this paper. We are also grateful to the referee for his/her valuable suggestions which we used to improve the exposition of this paper.

\section{Norm residue homomorphism and Galois modules}\label{sec:norm.residue}

One of the most exciting theorems from the last decade was the proof of the norm residue isomophism (previously, the Bloch-Kato conjecture).  To state this theorem, we remind the reader of some of important terms.  Throughout this section we assume that $p$ is a given prime number and that $F$ is a field which contains a primitive $p$th root of unity $\xi_p$.  We denote the separable closure of $F$ by $F_{\text{\tiny{sep}}}$ and the associated absolute Galois group by $G_F:=\Gal(\bar F_{\mbox{\tiny{sep}}}/F)$. When we speak of the cohomology of $F$, we mean the cohomology groups associated to the trivial $G_F$-module $\F_p$: $H^m(F):=H^m(G_F,\F_p)$. 

The Milnor $K$-groups $K_mF$ attached to $F$ are defined as $K_0F = \Z$, $K_1F = F^\times$ (the multiplicative group of $F$), and for $m>1$ as $$K_mF :=  (F^\times)^{\otimes m}/\langle a_1\otimes \cdots \otimes a_m: \exists 1\leq i < j \leq m \mbox{ so that } a_i+a_j=1\rangle.$$  Here $F^\times$ is viewed as an abelian group and the tensor product is over $\Z$.  In fact, we obtain a graded ring $K_*(F) = \oplus_{m=0}^\infty K_m(F)$ where the product is induced by tensor products.   We define the reduced Milnor $K$-groups as $k_mF:=K_mF/\langle p \rangle$, and they too form a graded ring which is also a graded vector space over $\F_p$.  An element of $k_mF$ which is represented by $f_1 \otimes \cdots \otimes f_m$ is written in the form $\{f_1,\cdots,f_m\}$.  For basic properties of Milnor $K$-theory we refer the reader to \cite{GS,Mi,Pf,Srin}.

Now since $\xi_p \in F$ we have $H^1(F) \simeq F^\times/F^{\times p}$, and so it is obvious that $H^1(F)$ and $k_1F$ are isomorphic.  The norm residue isomorphism says that this is true of higher reduced $K$-groups and cohomology as well, and that this isomorphism respects the underlying ring structures. 

The following theorem proved by M.~Rost and V.~Voevodsky (see \cite{V}) is a substantial advance in Galois cohomology which builds on previous work of J.~Arason, H.~Bass, S.~Bloch, R.~Elman, B.~Jacob, K.~Kato, T.Y.~Lam, A.~Merkurjev, J.~Milnor, A.~Suslin, J.~Tate and others.  See also \cite{P2} for a very nice survey concerning this theorem in the case $p=2$.

\begin{theorem*}[Norm residue isomorphism] The rings $k_*F$ and $H^*(F)$ are isomorphic via the map $h:k_*F \to H^*(F)$ defined by $h\left(\{f_1,\cdots,f_m\}\right)= (f_1) \cup \cdots \cup (f_m)$.
\end{theorem*}
Although the following proposition can be proved relatively simply, it is the first step in the long journey of proving the norm residue isomorphism.  J.~Milnor in \cite[p.~339--340]{Mi} says that
this result originated with Bass and Tate.  
\begin{proposition}[Bass-Tate Lemma]\label{prop:basstate} The map $h:k_*F \to H^*(F)$ is well-defined.
\end{proposition}
For this, we need to ensure that the defining relation on the level of $K$-theory maps to a coboundary in cohomology.  In other words, we need to show that if $a \in F^\times \setminus \{1\}$ is given, then the element $(a) \cup (1-a)$ is trivial in cohomology. Because the standard proof of this result is rather short and interesting, we shall include it here.  (See \cite[Lem.~8.1]{Srin} or \cite[Sec.~4.6]{GS}.) 
\begin{proof}
Because $(b^p) \cup (1-b^p) = 0$ for all $b \in F^\times$ such that $1 \neq b^p$, we shall assume that $a \not\in F^{\times p}$. 
Since we assume $\xi_p \in F$, we have the factorization $1-a = \prod_{i=0}^{p-1} \left(1-\xi_p^i \root{p}\of{a}\right)$ in the field $F(\root{p}\of{a})$.  This factorization, however, is equivalent to taking the norm $N_{F(\root{p}\of{a})/F}\left(1-\root{p}\of{a}\right)$, and so we have
\begin{align*}
(a) \cup (1-a) &= (a) \cup \left(N_{F(\root{p}\of{a})/F}\left(1-\root{p}\of{a}\right)\right).
\end{align*}
Now the projection formula (see \cite[Prop.~1.5.3(iv)]{NSW}) tells us that when $K/F$ is a Galois extension with $f \in F^\times$ and $k \in K^\times$, then the element $(f) \cup (N_{K/F}(k)) \in H^2(F)$ is equal to $\cor_{K/F}\left((f)\cup (k)\right)$, where here $\cor_{K/F}: H^2(K) \to H^2(F)$ is the corestriction map and the element $(f) \cup (k) \in H^2(K)$. Applied to our previous equation, the projection formula gives us
\begin{align*}
(a) \cup (1-a) &= (a) \cup \left(N_{F(\root{p}\of{a})/F}\left((1-\root{p}\of{a})\right)\right)\\
&=\cor_{F(\root{p}\of{a})/F}\left((a) \cup (1-\root{p}\of{a})\right)\\
&=\cor_{F(\root{p}\of{a})/F}\left(p\left(\root{p}\of{a}\right)\cup (1-\root{p}\of{a})\right)\\
&= \cor_{F(\root{p}\of{a})/F}\left(0\right) = 0.
\end{align*}
\end{proof}

This proof is certainly elegant, though the critical use of the projection formula prevents us from seeing the result in an ``elementary way" (i.e., one which exhibits a given cochain as a coboundary).  In the case $p=2$, Pfister gave just such an elementary proof of the Bass-Tate Lemma in \cite{Pf}. Our goal in this first section is to prove the vanishing of $(a) \cup (1-a)$ in two different ways: first by exploiting the known connection between the vanishing of $(x) \cup (y)$ and a particular Galois embedding problem concerning $H_{p^3}$; and then by giving an ``elementary" proof for $p>2$ akin to Pfister's proof for $p=2$.  In fact, in both proofs certain Galois extensions with Galois group $H_{p^3}$ play a key role.  If we consider $p=2$ instead, then the Galois extensions to be considered have Galois group either $\Z/4$ or the dihedral group of order $8$.

\subsection{A Proof of Bass-Tate via Galois embedding problems}\label{sec:vanishing.via.embedding}

The goal of this next subsection is to give an alternate proof of the vanishing of $(a) \cup (1-a)$.  We shall also assume that $p>2$, as the case $p=2$ is similar and no new insight is obtained when considering this case.  For the rest of this section we will assume that $a$ and $1-a$ are independent in $F^\times/F^{\times p}$.  This is a reasonable assumption as otherwise $(1-a) \cup (a)$ vanishes because the cup product is an anticommutative bilinear map $H^1(G_F,\F_p) \times H^1(G_F,\F_p) \to H^2(G_F,F_p)$ (see \cite[Ch.~1,Sec.~4]{NSW}).  It is possible to shorten our exposition below using \cite[Sec.~6.6]{JLY} or alternatively \cite[Th.~3.1]{Mich2} and \cite[Ch.~14,Sec.~4,Exer.~2(c)]{Pierce}, however we feel that it is instructive to present a detailed full exposition.

We recall (see \cite[p.~58]{Lbook}) that for $x,y \in F^\times$, the vanishing of $(x) \cup (y)$ is equivalent to the solvability of the Galois embedding problem
$$\xymatrix{1 \ar[r]& \F_p \ar[r]& H_{p^3} \ar[r]& \Z/p \times \Z/p \ar[r]& 1}$$ over $F(\root{p}\of{x},\root{p}\of{y})$.  Hence we will prove that $(a) \cup (1-a) = 0$ by finding an explicit $H_{p^3}$-extension of $F$ which contains $L:=F(\root{p}\of{a},\root{p}\of{1-a})$ as a quotient. We have already said that the criterion for solving this embedding problem over the field $F(\root{p}\of{a},\root{p}\of{1-a})$ is that $1-a \in N_{F(\root{p}\of{a})/F}(F(\root{p}\of{a})^\times)$, and  we have already observed this condition in our first proof of Bass-Tate. In a sense, then, we are done.  It will be profitable for us to carry the explanation out a bit more completely, however, as this will help us find an explicit representation of $(a)\cup (1-a)$ as a coboundary, and because it is intimately connected to the results we present in subsequent sections.  

First, we establish some notation.  We know that for $L=F(\root{p}\of{a},\root{p}\of{1-a})$ we have $\Gal(L/F) \simeq \Z/p \times \Z/p$; we will write $\sigma$ and $\tau$ for generators of this group which are dual to $a$ and $1-a$, respectively (e.g., $\sigma(\root{p}\of{a}) = \xi_p \root{p}\of{a}$, yet $\sigma$ acts trivially on $\root{p}\of{1-a}$). Now consider $\alpha = 1-\root{p}\of{a} \in F(\root{p}\of{a})$.  We have already seen that $$1-a = \prod_{i=0}^{p-1} \left(1-\xi_p^i \root{p}\of{a}\right) = N_{F(\root{p}\of{a})/F}\left(\alpha\right).$$ We define $\beta = (\sigma-1)^{p-2} (\alpha)$ and claim that the desired $H_{p^3}$-extension of $L/F$ is $L(\root{p}\of{\beta})/F$.  

First, observe that the $F$-conjugates of $\root{p}\of{\beta}$ are the $p$th roots of $\beta$ under the action of $\sigma$ and $\tau$.   Since $\beta \in F(\root{p}\of{a})$, we have $\tau\left(\beta\right) = \beta$.  On the other hand, the action of $\sigma$ on $\beta$ is nontrivial; because $1+\sigma+\cdots + \sigma^{p-1} \equiv (\sigma-1)^{p-1} \mod{p}$, we have
\begin{align*}\sigma(\beta) &= \sigma\left((\sigma-1)^{p-2}\alpha\right) = \beta (\sigma-1)^{p-1}\alpha \\&\equiv \beta N_{F(\root{p}\of{a})/F}(\alpha) = \beta (1-a) \mod{F(\root{p}\of{a})^{\times p}}.
\end{align*}
In either case, though, the $p$th roots of $\tau(\beta)$ and $\sigma(\beta)$ are contained in $L(\root{p}\of{\beta})$, and hence $L(\root{p}\of{\beta})$ is Galois.

To prove $\Gal(L(\root{p}\of{\beta})/F) \simeq H_{p^3}$, begin by noting that this extension is degree $p^3$.  Hence all we must do is show that this group is noncommutative and that elements have order at most $p$.  Let $\hat \sigma$ and $\hat \tau$ be lifts of $\sigma$ and $\tau$ to $\Gal(L(\root{p}\of{\beta})/F)$; it will be enough for us to show that $\hat \sigma \hat \tau \neq \hat \tau \hat \sigma$ and that $\hat \sigma^p = \hat \tau^p = \mbox{id}$.  To arrive at these results, we will only need to investigate actions on $\root{p}\of{\beta}$ since we already know how $\hat \tau$ and $\hat \sigma$ act on $\root{p}\of{a}$ and $\root{p}\of{1-a}$.  

We have already seen that $\hat \tau(\beta) = \beta$ and $\hat \sigma(\beta) = \beta (1-a) k^p$ for some $k \in F(\root{p}\of{a})$.  By extracting $p$th roots in these equations, we therefore have
\begin{equation}\label{eq:tau.and.sigma.relations}
\begin{split}
\hat \sigma(\root{p}\of{\beta}) &= \xi_p^x~\root{p}\of{\beta}~\root{p}\of{1-a}~k\\
\hat \tau(\root{p}\of{\beta}) &= \xi_p^y~\root{p}\of{\beta}
\end{split}
\end{equation} 
for some $x,y \in \Z$. With these identities in hand, it is easy to see that $\hat \tau$ and $\hat \sigma$ do not commute:
$$\xi_p \hat \sigma \hat \tau(\root{p}\of{\beta}) =  \hat \tau \hat \sigma(\root{p}\of{\beta}).$$  Furthermore by iteratively applying the second part of equation (\ref{eq:tau.and.sigma.relations}), one recovers $\hat \tau^{p}(\root{p}\of{\beta})=\root{p}\of{\beta}$.  Hence $\hat \tau^p=\mbox{id}$.  On the other hand, 
\begin{align*}
\frac{\hat \sigma^p(\root{p}\of{\beta})}{\root{p}\of{\beta}} &= (\hat \sigma-1)(1+\hat \sigma+\hat \sigma^2 + \cdots + \hat \sigma^{p-1})(\root{p}\of{\beta})\\
&=(\hat \sigma-1) \left(\xi_p^z \root{p}\of{(1+\sigma+\sigma^2 + \cdots + \sigma^{p-1})\beta}\right)\\
&=(\hat \sigma-1)\left(\xi_p^z \root{p}\of{N_{F(\root{p}\of{a})/F}(\beta)}\right)\\
&=(\hat \sigma-1)\left(\xi_p^z \root{p}\of{1}\right)\\
&=1.
\end{align*}In this case the second-to-last equality comes from the fact that $\beta$ is in the image of $\sigma-1$, and the last equality follows because $\xi_p \in F$ by assumption.

\subsection{Exhibiting $(1-a) \cup (a)$ explicitly as a coboundary}


In \cite{Pf}, Pfister gives a very interesting proof of the vanishing of $(a) \cup (1-a)$ in the case when $p=2$.  Pfister's proof uses nothing but basic definitions from Galois cohomology; he shows that $(a) \cup (1-a)=0$ by explicitly exhibiting it as a coboundary.  His proof, however, is unmotivated.  Where does his choice of coboundary come from?  Also, one would like to produce small Galois extensions where one can already define $(a) \cup (1-a)$ and show that it vanishes there.   In our proof below, we explicitly exhibit $(1-a) \cup (a)$ as a coboundary; there is no significant difference between this and $(a) \cup (1-a)$ considering the symmetry between $a$ and $1-a$ and the anti-commutative property of the cup product (see \cite[Prop.~1.4.4]{NSW}).  We feel that the vanishing of $(1-a)\cup(a)$ is in fact a precursor of the $n$-vanishing Massey conjecture for $n \geq 3$. (See \cite{Ef,HW,MT1,MT2,MT3} for related material.)

In this section we provide motivation for Pfister's proof, and in the process we both extend it to the case $p>2$ and show that $(1-a)\cup(a)$ can be defined on a Galois extension of degree $p^3$ on which it vanishes.  Though thematically similar, there are some small distinctions between the proof when $p>2$ and $p=2$; for this reason, we shall focus on the case $p>2$ and only indicate the necessary changes for $p=2$.  So assume now that $p>2$.  

Before diving into the details, we make a short road map for our argument.  Let $\mathbf{p}:F_{\text{\tiny{sep}}}^\times \to F_{\text{\tiny{sep}}}^\times$ be the map $\mathbf{p}(f) = f^p$; the kernel of this map is $\langle \xi_p \rangle \simeq \Z/p$, and we write $\eta$ for the associated embedding of $\Z/p$ into $F_{\text{\tiny{sep}}}^\times$.  The short exact sequence $$\xymatrix{1 \ar[r] & \Z/p \ar[r]^\eta & F_{\tiny{\text{sep}}}^\times \ar[r]^{\mathbf{p}} & F_{\tiny{\text{sep}}}^\times \ar[r] & 1}$$ induces the exact sequence on cohomology
\begin{equation}\label{eq:long.exact.sequence}
\xymatrix{H^1(G_F,F_{\text{\tiny{sep}}}^\times) \ar[r] & H^1(G_F,F_{\text{\tiny{sep}}}^\times) \ar[r]^\delta & H^2(G_F,\Z/p) \ar[r]^{\eta_*} & H^2(G_F,F_{\text{\tiny{sep}}}^\times).}
\end{equation}
Although $H^1(G_F,F_{\text{\tiny{sep}}}^\times) = \{0\}$ by Hilbert 90, it is convenient to keep this group in our picture since we need concrete $1$-cochains of $G_F$ in $F_{\text{\tiny{sep}}}^\times$.  Using the Galois theory of $p^3$-extensions, we produce a $1$-cochain $N \in C^1(G_F,F_{\text{\tiny{sep}}}^\times)$ such that $d^1(N) = \eta_*\left((1-a) \cup (a)\right)$, and hence $\eta_*\left((1-a) \cup (a)\right) = 0$ in $H^2(G_F,F_{\text{\tiny{sep}}}^\times)$.  Because $\eta_*$ is injective, this implies that $(1-a) \cup (a) = (0) \in H^2(G_F,\Z/p)$.  However, in order to produce a specific $2$-boundary for the cochain representing $(1-a) \cup (a)$, we chase sequence (\ref{eq:long.exact.sequence}) and modify $N$ by another $1$-cochain $M$ so that $d^1(M)$ is trivial but $N/M$ takes values in $\Z/p$.

As in the previous section, assume that $a$ and $1-a$ represent classes which are linearly independent in the $\F_p$-vector space $F^\times/F^{\times p}$.  We again write $L = F(\root{p}\of{a},\root{p}\of{1-a})$.  Thus $L/F$ is a Galois extension and $\Gal(L/F) \simeq \Z/p \times \Z/p$.  We fix a specific primitive $p$th root of unity $\xi_p$.  This will give us an isomorphism $\iota: \langle \xi_p \rangle \to \Z/p$; in this way we can identify $\langle \xi_p \rangle$ with $\Z/p$ via $\iota$.  Following the convention from section \ref{sec:vanishing.via.embedding}, we write $\sigma$ and $\tau$ for the generators of $\Gal(L/F)$ that are dual to $\root{p}\of{a}$ and $\root{p}\of{1-a}$, respectively.

Observe that $(1-a) \in H^1(G_F,\Z/p)$ is represented by $\gamma \mapsto \iota\left(\gamma(\root{p}\of{1-a})/\root{p}\of{1-a}\right)$, and similarly $(a) \in H^1(G_F,\Z/p)$ is represented by $\gamma \mapsto 
\iota\left(\gamma(\root{p}\of{a})/\root{p}\of{a}\right)$.  Let $\gamma_1,\gamma_2 \in G_F$ restrict to $\tau^i \sigma^k,\tau^l \sigma^j\in \Gal(L/F)$, respectively;  we have (see \cite[Prop.~1.4.8]{NSW}) that $(1-a) \cup (a) \in Z^2(G_F,\Z/p)$ is the function $$(\gamma_1,\gamma_2) \mapsto \iota\left(\frac{\gamma_1(\root{p}\of{1-a})}{\root{p}\of{1-a}}~\frac{\gamma_2(\root{p}\of{a})}{\root{p}\of{a}}\right) = ij \mod{p}.$$  

We now proceed to produce  a $1$-cochain $h \in C^1(G_F,\Z/p)$ so that $d^1 h(\gamma_1,\gamma_2) = ij \mod{p}$ (where $d^1:C^1(G_F,\Z/p) \to Z^2(G_F,\Z/p)$ is the usual coboundary map).  To begin, set $\alpha = 1-\root{p}\of{a} \in F(\root{p}\of{a})$ and $\beta = \alpha^{p-1} \sigma(\alpha^{p-2}) \dots\sigma^{p-2}(\alpha)$.  (Note that this element $\beta$ isn't identical to the element $\beta$ in section \ref{sec:vanishing.via.embedding}; see also \cite[p.~162]{JLY}.) Note that for the element $x = \root{p}\of{1-a}/\alpha \in L$ we have $\sigma(\beta)/\beta = x^p$ and 
\begin{equation}\label{eq:norm.of.x}x\sigma(x)\cdots \sigma^{p-1}(x) = \frac{\left(\root{p}\of{1-a}\right)^p}{N_{F(\root{p}\of{a})/F}(\alpha)} = \frac{1-a}{1-a} = 1.\end{equation}  We also have $\tau(x) = \xi_p x$ and $\tau(\beta) = \beta$.  

Now let us consider the possible images of $\beta$ under $\Gal(L/F)$.  First, one proves by induction that for any $k \in \{1,\dots,p-1\}$, we have $$\sigma^k(\beta) = \left(\sigma^{k-1}(x)\sigma^{k-2}(x)\cdots \sigma(x) x\right)^p \beta.$$ Then since all elements of $\Gal(L/F)$ can be written in the form $\tau^i \sigma^k$, and since we already know that $\tau$ acts trivially on $x^p$ and $\beta$, this gives all possible images of $\beta$ under $\Gal(L/F)$. 
In particular, this means that if $\gamma \in G_F$ restricts to $\tau^i \sigma^k$, then there exists a unique $c \in \{0,1,\dots,p-1\}$ such that $\gamma\left(\root{p}\of{\beta}\right) = \xi_p^c \root{p}\of{\beta} \prod_{l=0}^{k-1} \sigma^l(x)$.  We use this last equation to define two functions from $G_F$ to $F_{\tiny{\text{sep}}}^\times$:  
\begin{align*}
N(\gamma) = \prod_{l=0}^{k-1} \sigma^l(x) \quad \mbox{ and } \quad
M(\gamma) =\frac{\gamma(\root{p}\of{\beta})}{\root{p}\of{\beta}} = \xi_p^c  \prod_{l=0}^{k-1} \sigma^l(x).\\
\end{align*}
Note that $M \in Z^1(G_F,F_{\tiny{\text{sep}}}^\times)$ since it is conspicuously in $B^1(G_F,F_{\tiny{\text{sep}}}^\times)$.  
Observe that using equation (\ref{eq:norm.of.x}), one can show that the function $N \in C^1(G_F,F_{\tiny{\text{sep}}}^\times)$ is well-defined.
Finally, observe that since $$\beta \equiv (\sigma-1)^{p-2}(\alpha) \in F(\root{p}\of{a})^\times/F(\root{p}\of{a})^{\times p},$$  section \ref{sec:vanishing.via.embedding} tells us that $L(\root{p}\of{\beta})/F$ is Galois with $\Gal(L(\root{p}\of{\beta})/F) \simeq H_{p^3}$.  Hence $M$ and $N$ are defined on the $H_{p^3}$-extension $L(\root{p}\of{\beta})/F$.

As a consequence, we have a $1$-cochain $h = \iota_*\left(N/M\right) \in C^1(\Gal(L(\root{p}\of{\beta})/F),\Z/p)$.  We show that --- as cocycles --- we have $\text{inf}\left(d^1 h\right) = (1-a) \cup (a)$.  To do this, let us assume that $\gamma_1$ and $\gamma_2$ are elements of $G_F$ with restrictions in $L/F$ of the form $\tau^i \sigma^k$ and $\tau^l\sigma^j$; our goal will be to show $\text{inf}\left(d^1 h\right)(\gamma_1,\gamma_2) = ij \mod{p}$.  Since $M \in B^1(G_F,F_{\tiny{\text{sep}}})$ we know $d^1 M$ is trivial.  Therefore we have
\begin{align*}
\text{inf}\left(d^1 \frac{N}{M}\right)(\gamma_1,\gamma_2) =  \iota\left(\frac{N(\gamma_2)^{\gamma_1}N(\gamma_1)}{N(\gamma_1\gamma_2)}\frac{M(\gamma_1\gamma_2)}{M(\gamma_2)^{\gamma_1}M(\gamma_1)}\right) = \iota\left(\frac{N(\gamma_2)^{\gamma_1}N(\gamma_1)}{N(\gamma_1\gamma_2)}\right).
\end{align*}  By definition we have 
\begin{align*}
N(\gamma_1) = \prod_{l=0}^{k-1} \sigma^l(x), \quad N(\gamma_2) = \prod_{l=0}^{j-1}\sigma^{l}(x), \quad \mbox{ and } \quad N(\gamma_1\gamma_2) & = \prod_{l=0}^{j+k-1}\sigma^{l}(x).
\end{align*} Using the relation $\tau(x) = \xi_px$ and the fact that $\sigma$ and $\tau$ commute in $\Gal(L/F)$, we recover
\begin{align*}
\frac{N(\gamma_2)^{\gamma_1} N(\gamma_1)}{N(\gamma_1\gamma_2)} &= 
\left(\prod_{l=0}^{j-1} \sigma^l(x)\right)^{\tau^i \sigma^k}\left(\prod_{l=0}^{k-1}\sigma^l(x)\right)\left(\prod_{l=0}^{j+k-1}\sigma^l(x)\right)^{-1} 
\\&= \left( \prod_{l=0}^{j-1} \sigma^l(x^{\tau^i})\right)^{\sigma^k}\left(\prod_{l=0}^{i-1}\sigma^l(x)\right)\left(\prod_{l=0}^{j+k-1}\sigma^l(x)\right)^{-1} = \xi_{p}^{ij}.
\end{align*}
This gives the desired result.  

The case $p=2$ is nearly identical with the case $p>2$, except that in this case one should consider also the case when $1-a$ and $a$ belong to the same class in $F^\times/F^{\times 2}$.  In this case $L = F(\root\of{a},\root\of{1-a}) = F(\root\of{a})$ is a quadratic extension and $L(\root\of{1-\root\of{a}})/F$  is cyclic of degree $4$.  When $1-a$ and $a$ are independent modulo $F^{\times 2}$, the extension $L(\root\of{\alpha})/F$ is a dihedral extension of degree $8$.  It is now straightforward to check that the construction above provides a $1$-cochain $h \in C^1(G_F,\F_2)$ such that $d^1h = (a) \cup (1-a)$ (with $(a)$ and $(1-a)$ viewed as $1$-cochains and $(a) \cup (1-a)$ viewed as an element of $Z^2(G_F,\F_2)$).  Moreover we recover in this way the cochain $h$ introduced by Pfister in \cite[p.~275]{Pf}.

\section{Recasting the embedding problem}\label{sec:recasting}

We continue to assume that $p$ is a prime number, but for the rest of the paper we assume that $p>2$.  We assume that $F$ is a field and $K/F$ an extension with Galois group isomorphic to $\Z/p^n$.  We no longer make assumptions on the characteristic of $F$ or on the roots of unity it contains.

Traditionally, theorems concerning the realizability of $H_{p^3}$ and $M_{p^3}$ as Galois groups are studied from the perspective of embedding problems that arise from group extensions of $\Z/p \times \Z/p$ by $\Z/p$.  In this second part of our paper, we revisit the realizability of these groups as Galois groups by instead studying them as a special case of the family of embedding problems 
\begin{equation*}
\xymatrix{1 \ar[r]& \left(\Z/p\right)^{\oplus \ell} \ar[r]& G \ar[r] & \Z/p^n \ar[r] & 1}.
\end{equation*}  (Both $H_{p^3}$ and $M_{p^3}$ occur as short exact sequences for $\ell=2$ and $n=1$.)  By studying appropriate $\F_p$-vector spaces as modules over Galois groups, we are able to associate the realizations of these groups as Galois groups to the appearance of modules of a certain type within the classical parameterizing spaces of elementary $p$-abelian extensions; by using the recently computed module structures for these spaces, we can revisit the known results concerning these groups from this module-theoretic perspective, and furthermore exhibit them as part of a broader phenomena.  

We begin by establishing some notation and reminding the reader of some module-theoretic machinery.  We write $G_n$ for the group $\Z/p^n$, and we use $\sigma$ to denote a generator of this group (the particular $n$ corresponding to $\sigma$ will be clear from the context).  

$\F_p[G_n]$ is a local ring with maximal ideal $\langle \sigma-1 \rangle$.  We define a homomorphism $\psi:\F_p[t] \to \F_p[G_n]$ by $\psi(t) = \sigma-1$.  Now $\psi$ is a surjective map, and $t^{p^n} \in \ker(\psi)$ because $$1 = \sigma^{p^n} = \left((\sigma-1)+1\right)^{p^n} = (\sigma-1)^{p^n}+1.$$  Counting dimensions over $\F_p$, we conclude that $\F_p[t]/\langle t^{p^n}\rangle \simeq \F_p[G_n]$.  Hence we define a ``valuation-like" map $v:\F_p[G_n] \to \mathcal{L}_n$, where $\mathcal{L}_n = \{0,1,\cdots,p^n-1\} \cup \{\infty\}$ is a set endowed with a binary operation $*$ defined by $$i*j = \left\{\begin{array}{ll}i+j,&\mbox{ if }i,j \neq \infty \mbox{ and }i+j \leq p^n-1\\\infty, &\mbox{ if }i=\infty \mbox{ or }j=\infty \mbox{ or }i+j > p^n-1.\end{array}\right.$$  For nonzero $f \in \F_p[G_n]$, we will write $v(f)$ for the maximum value $i \in \mathcal{L}_n \setminus \{\infty\}$ satisfying $f \in \langle (\sigma-1)^i \rangle$, and we set $v(0) = \infty$.  Then we have $v(fg) = v(f)*v(g)$, and $v(f+g) \geq \min\{v(f),v(g)\}$ (following the usual convention that $\infty>i$ for all $i \in \mathcal{L}_n\setminus\{\infty\}$).  If $M$ is an $\F_p[G_n]$-module (written additively) and $\gamma \in M$, then the smallest positive integer $\ell$ such that $(\sigma-1)^{\ell}\gamma = 0$ coincides with the $\F_p$-dimension of $\langle \gamma \rangle$; we write $\ell(\gamma)$ for this quantity.  

If $\gamma_1,\gamma_2 \in M$ and $\mu = \max\{\ell(\gamma_1),\ell(\gamma_2)\}$ we see that $$(\sigma-1)^{\mu} (\gamma_1 + \gamma_2) = (\sigma-1)^\mu\gamma_1 + (\sigma-1)^\mu \gamma_2 = 0.$$  Moreover, if $\ell(\gamma_1)<\ell(\gamma_2)$, then for any $\ell(\gamma_1) \leq v < \ell(\gamma_2)$ we have $$(\sigma-1)^v(\gamma_1+\gamma_2)= (\sigma-1)^v\gamma_1 +(\sigma-1)^v\gamma_2 = (\sigma-1)^v\gamma_2 \neq 0.$$  Hence we see that $\ell(\gamma_1+\gamma_2) \leq \max\{\ell(\gamma_1),\ell(\gamma_2)\}$, with equality if $\ell(\gamma_1) \neq \ell(\gamma_2)$.  In what follows, we will refer to this as the ultrametric property.

For $1 \leq \ell \leq p^n$, we define the $\F_p[G_n]$-module $A_\ell$ as $\F_p[G_n]/\langle(\sigma-1)^\ell\rangle$; the action of $\sigma$ on $A_\ell$ is simply multiplication by $\sigma$.  Each of modules from $\{A_{\ell}\}_{\ell=1}^{p^n}$  is indecomposable, and any indecomposable $\F_p[G_n]$-module is isomorphic to some $A_\ell$.  For any $\F_p[G_n]$-module $A$ there is a unique (unordered) collection of positive integers $\{\ell_i\}_{i \in \mathcal{I}}$ (where possibly $\ell_i = \ell_j$ for $i \neq j$) so that $A \simeq \bigoplus_{i\in\mathcal{I}} A_{\ell_i}$.  (For more details, see \cite[Sec.~1.1]{MSS1}, \cite[Sec.~2.3]{LMSSembed}, or \cite{AnFu}.)  

\begin{remark*}
The $A_\ell$'s are exactly the indecomposable representations of the cyclic group $G_n$ over the field $\F_p$.
\end{remark*}

\begin{remark*}
By a slight abuse of notation, we denote elements of $A_\ell$ as elements of $\F_p[G_n]$.  In these cases, it should be understood that we intend the given element modulo the ideal of $\F_p[G_n]$ generated by $(\sigma-1)^\ell$.
\end{remark*}

\subsection{Parameterizing spaces of elementary $p$-abelian extensions}

When $K$ (and therefore $F$) contains a primitive $p$th root of unity $\xi_p$,  Kummer theory tells us that elementary $p$-abelian extensions of $K$ correspond to $\F_p$-subspaces of $K^\times/K^{\times p}$; these extensions are additionally Galois over $F$ if and only if they are modules over the group ring $\F_p[G_n]$.  In this case, we define $J(K) = K^\times/K^{\times p}$.  

When $K$ has characteristic $p$, Artin-Schreier theory gives a correspondence between elementary $p$-abelian extensions of $K$ and $\F_p$-subspaces of $K/\wp(K)$, where $\wp(K) = \{k^p - k: k \in K\}$.  Again, such an extension is Galois over $F$ if and only if the corresponding $\F_p$-space is a module over $\F_p[G_n]$.  In this case we define $J(K) = K/\wp(K)$.

Finally, if $K$ has characteristic different from $p$ but $\xi_p \not\in K$ (and therefore $\xi_p \not\in F$), then elementary $p$-abelian extensions of $K$ correspond to $\F_p$-subspaces of a particular eigenspace of $K(\xi_p)^\times/K(\xi_p)^{\times p}$.  Specifically, if $\tau$ is a generator for $\Gal(K(\xi_p)/K)$ and $\tau(\xi_p) = \xi_p^t$, then the space parametrizing elementary $p$-abelian extensions of $K$ is the subspace on which $\tau$ acts as exponentiation by $t$ (the ``$t$-eigenmodule").  As before, such an extension is Galois over $F$ if and only if the corresponding $\F_p$-space is a module over the group ring $\F_p[G_n]$ (where here we identify $\Gal(K/F)$ and $\Gal(K(\xi_p)/F(\xi_p))$).  In this case, we define $J(K)$ as the $t$-eigenmodule of $K(\xi_p)^\times/K(\xi_p)^{\times p}$.  (These parameterizing spaces are also reviewed in \cite{Wat}.)

It is worth noting that in this latter case, one can describe a morphism $\mathcal{T}$  which projects subspaces of $K(\xi_p)^\times/K(\xi_p)^{\times p}$ to $J(K)$.  Let $s=[K(\xi_p):K]$, and note that $1<s \leq p-1$.  Choose $z \in \Z$ so that $zst^{s-1} \equiv 1 \mod{p}$, and --- again using the notation $\tau(\xi_p) = \xi_p^t$ --- set \begin{equation}\label{eq:projector.for.descent}\mathcal{T} = z \sum_{i=1}^s t^{s-i} \tau^{i-1} \in \Z[\langle \tau \rangle].\end{equation}  Notice that $(t-\tau)\mathcal{T} \equiv 0 \mod{p}$, and so the image of $\mathcal{T}$ is contained in the $t$-eigenspace for $\tau$.  Conversely, if an element is in the $t$-eigenspace for $\tau$, then $\mathcal{T}$ acts as the identity.  Hence we have $\mathcal{T}$ projects $\F_p$-subspaces of $K(\xi_p)^\times/K(\xi_p)^{\times p}$ onto $J(K)$.  (For more details, see \cite[Sec.~4]{MSauto} or the proof of \cite[Thm.~2]{MSSauto}.)

Before moving on, we make a brief comment on notation.  Our goal is to prove statements about certain embedding problems within the uniform framework provided by $J(K)$.  For this reason, whenever we discuss $J(K)$ we will assume it has an underlying additive structure (and therefore $\Gal(K/F)$ acts multiplicatively).  


\subsection{Module structures and Galois groups}

Because $\Gal(K/F)$ induces an action on $J(K)$, it is natural to consider what this additional structure tells us about elementary $p$-abelian extensions of $K$.   The first answer to this question was given by Waterhouse in \cite{Wat}, where he considered cyclic submodules of $J(K)$ when $\textrm{char}(K) \neq p$.  The question was answered for non-cylic modules, as well as in the case $\textrm{char}(K) = p$, by the authors in \cite{MSSauto,Schultz2}; there the $\F_p[G_n]$-module $J(K)$ is exhibited as a parameterizing space for solutions to embedding problems over $K/F$ which arise from extensions of $G_n$ by elementary $p$-abelian groups.   Since we are focusing on groups related to $H_{p^3}$ and $M_{p^3}$ in this paper, it will be sufficient for us to focus our attention on cyclic $\F_p[G_n]$-submodules. 

So let $L/K$ be an elementary $p$-abelian extension which corresponds to a cyclic $\F_p[G_n]$-submodule $\langle \gamma \rangle \subseteq J(K)$.  Then $L/F$ is Galois and $\Gal(L/F)$ is an extension of $\Gal(K/F) \simeq G_n$ by $\Gal(L/K)$.  Using the appropriate parameterizing theory (e.g., Kummer theory if $\textrm{char}(K) \neq p$ and $\xi_p \in K$), one can show that there is an equivariant pairing $\Gal(L/K) \times \langle \gamma \rangle \to \F_p$, and hence $\Gal(L/K)$ is dual to $\langle \gamma \rangle$.  One can show that $\F_p[G_n]$-modules are self-dual, and so we conclude that $\Gal(L/F)$ is an extension of $\Gal(K/F) \simeq G_n$ by $\langle \gamma \rangle$.

With this in mind, we now describe the possible extensions of $G_n$ by a cyclic $\F_p[G_n]$-module.

\begin{proposition}[{\cite[Thm.~2]{Wat}}]
There is only one group extension of $G_n$ by $A_{p^n}$, namely the semi-direct product $A_{p^n}\rtimes G_n$.  For $1 \leq i < p^n$ there are two possible group extensions of $G_n$ by $A_i$.  One of them is the semi-direct product $A_i \rtimes G_n$, where we have $$(f_1,\sigma^{j_1})(f_2,\sigma^{j_2}) = (f_1 + \sigma^{j_1}f_2,\sigma^{j_1+j_2}).$$

The second extension of $G_n$ by $A_i$ will be written $A_i \bullet G_n$; the elements of this group again come from $A_i \times G_n$, but the operation is given by 
$$(f_1,\sigma^{j_1})(f_2,\sigma^{j_2}) = \left\{\begin{array}{ll}(f_1+\sigma^{j_1}f_2,\sigma^{j_1+j_2})&,\mbox{ if }j_1+j_2<p^n\\(f_1+\sigma^{j_1}f_2 + (\sigma-1)^{i-1},\sigma^{j_1+j_2})&,\mbox{ if }j_1+j_2 \geq p^n.\end{array}\right.$$ (Here the numbers $j_1$ and $j_2$ are taken from $\{0,\cdots,p^n-1\}$.)
\end{proposition}

\begin{example}
Let $\sigma,\tau \in H_{p^3}$ be nontrivial elements which generate $H_{p^3}$.  Then we have $[\sigma,\tau] = \sigma\tau\sigma^{-1}\tau^{-1}$ is an element of order $p$ which generates $Z(H_{p^3})$.  Now consider the $\F_p[\langle \sigma \rangle]$-module $B$ which is the $\F_p$-span of $\{\tau,[\sigma,\tau]\}$ (where the action of $\sigma$ is by conjugation).  The computation $(\sigma-1)\cdot \tau = \sigma\tau\sigma^{-1}\tau^{-1} = [\sigma,\tau]$ shows that $\langle \tau \rangle = B$ as an $\F_p[\langle \sigma \rangle]$-module, and so $B \simeq A_2$.  Because $\langle \sigma \rangle \cap B =\{1\}$ and $\langle \sigma \rangle \simeq G_1$, we conclude that $H_{p^3} \simeq A_2 \rtimes G_1$.  Phrased in a slightly different language, this example tells us that if $N$ is any subgroup of $H_{p^3}$ with $|N| = p^2$ and $Q:=H_{p^3}/N$, then $N \simeq A_2$ as an $\F_p[Q]$-module and $H_{p^3} \simeq N \rtimes Q$.

Now recall that $M_{p^3} = \left\langle y,x~ |~ y^{p^2} = x^{p} = 1, [x,y] = y^p\right\rangle = \langle y \rangle \rtimes \langle x \rangle$.  The subgroups $\langle y^k x \rangle$ for $k \in \{1, \cdots, p-1\}$, together with the subgroup $\langle y \rangle$, provide $p$ distinct subgroups isomorphic to $\Z/p^2$.  On the other hand, we claim that the subgroup $N=\langle y^p , x \rangle$ can be the only subgroup of $M_{p^3}$ isomorphic to $\Z/p \times \Z/p$. To see this, note that the $p$ subgroups isomorphic to $\Z/p^2$ provide us with $p\cdot\phi(p^2) = p^3-p^2$ elements of order $p^2$, and so the number of elements of order less than $p^2$ within $M_{p^3}$ is at most $p^2$.  Since $N$ already contains $p^2$ elements of this type, it can be the only subgroup isomorphic to $\Z/p \times \Z/p$.  On the other hand, a direct computation shows that $[y^p,x]=1$, so that $N \simeq \Z/p \times \Z/p$.

Now let $Q = M_{p^3}/N  = \langle yN \rangle \simeq G_1$.  We have that $N$ is generated by $x$ under the action of $yN$, and hence $N \simeq A_2$ as an $\F_p[G_1]$-module. In this case, however, one has $M_{p^3} \simeq A_2 \bullet G_1$ (one can verify this directly, or simply note that $A_2 \rtimes G_1$ has no elements of order $p^2$).
\end{example}

\begin{example}
Let $n\geq 3$, and consider the group $A_2 \bullet G_{n-2}$.  Using the notation from the previous proposition, we have
\begin{align*}
(0,\sigma)^{p^{n-2}} &= (\sigma-1,1)\\
(0,\sigma)^{p^{n-1}} &= \left((0,\sigma)^{p^{n-2}}\right)^p = (\sigma-1,1)^p = (0,1).
\end{align*}
Hence $A_2 \bullet G_{n-2}$ is a nonabelian $p$-group with order $p^n$ that contains an element of order $p^{n-1}$.  Since we've already observed that there is only one such $p$-group --- which we previously called $M_{p^n}$ --- we must have $M_{p^n} \simeq A_2 \bullet G_{n-2}$.
\end{example}

If $L/K$ corresponds to $\langle \gamma \rangle$, all that is left to do is determine which extension of $G_n$ by $\langle \gamma \rangle$ corresponds to $\Gal(L/F)$.  To do this, one uses the so-called index function.  The index is a function $e:J(K) \cap \ker\left((\sigma-1)^{p^n-1}\right) \to \F_p$ defined by 
$$e(\gamma) = \left\{\begin{array}{ll}\root{p}\of{N_{\hat K/\hat F}(\gamma)}^{\sigma-1},&\mbox{ if }\textrm{char}(K) \neq p\\(\sigma-1)\rho(Tr_{K/F}(\gamma)),&\mbox{ if }\textrm{char}(K) = p.\end{array}\right.$$ (In the first case we have identified $\mu_p$ with $\F_p$ by selecting a particular root of unity $\xi_p$ to act as a generator of $\mu_p$.)  With the index in hand, we have the following

\begin{proposition}[\cite{Schultz2}, Thm.~4.4]\label{prop:modules.as.galois.groups}
Suppose $\Gal(K/F) \simeq G_n$.  Solutions to the embedding problem $A_{p^n} \rtimes G_n \twoheadrightarrow G_n$ are in correspondence with submodules $\langle \gamma \rangle \subseteq J(K)$ such that $\ell(\gamma) = p^n$.  For $i<p^n$, solutions to the embedding problem $A_i \rtimes G_n \twoheadrightarrow G_n$ over $K/F$ are in correspondence with the submodules $\langle \gamma \rangle \subseteq J(K)$ such that $\ell(\gamma) = i$ and $e(\gamma) = 0$; solutions to the embedding problem $A_i \bullet G_n \twoheadrightarrow G_n$ over $K/F$ are in correspondence with the submodules $\langle \gamma \rangle \subseteq J(K)$ such that $\ell(\gamma) = i$ and $e(\gamma) \neq 0$.
\end{proposition}

\begin{example}\label{ex:hp3.and.mp3.in.module.language}
Suppose that $L/F$ has $\Gal(L/F) \simeq H_{p^3}$, and let $K/F$ be any $\Z/p$-subextension.  Then $L/K$ is Galois with $\Gal(L/K) \simeq \Z/p \times \Z/p$, and hence $L$ corresponds to some submodule $M \subseteq J(K)$.  Since $\Gal(L/F)$ is nonabelian it must be the case that $M \not\simeq A_1 \oplus A_1$, and so $M \simeq A_2$.  If $\langle \gamma \rangle = M$, we must have $e(\gamma) = 0$ by the previous proposition.  Conversely, if $K/F$ is an extension with $\Gal(K/F) \simeq \Z/p$ and $\langle \gamma \rangle \subseteq J(K)$ satisfies $\ell(\gamma)=2$ and $e(\gamma) = 0$, then $\langle \gamma \rangle$ corresponds to an extension $L/K$ with $\Gal(L/F) \simeq H_{p^3}$.

Suppose now that $L/F$ has $\Gal(L/F) \simeq M_{p^3}$.  Then there is a unique subextension $K/F$ with $\Gal(K/F) \simeq \Z/p$ and $\Gal(L/K) \simeq \Z/p \times \Z/p$.  Within $J(K)$ there exists a submodule $M$ corresponding to $L/K$, and again it must be the case that $M =\langle \gamma \rangle$ with $\ell(\gamma) = 2$ and $e(\gamma) \neq 0$.  Conversely, if $K/F$ is an extension with $\Gal(K/F) \simeq \Z/p$ and $\langle \gamma \rangle \subseteq J(K)$ satisfies $\ell(\gamma) = 2$ and $e(\gamma) \neq 0$, then $\langle \gamma\rangle$ corresponds to an extension $L/F$ containing $K$ with $\Gal(L/F) \simeq M_{p^3}$.
\end{example}

The structure of $J(K)$ was computed when $\xi_p \in K$ in \cite[Thm.~2] {MSS1}, when $\xi_p \not\in K$ but $\textrm{char}(K) \neq p$ in \cite[Thm.~2]{MSSauto}, and when $\textrm{char}(K) = p$ in \cite[Prop.~6.2]{Schultz2}.  (Note that in the case of $\textrm{char}(K) = p$ and $\xi_p \not\in F$,  there is a module decomposition for $K^\times/K^{\times p}$ provided by \cite[Thm.~1]{MSS1}; in this case, however, $J(K) \neq K^\times/K^{\times p}$, so this is not the decomposition we provide below.)  Though there are some distinctions between the structure of $J(K)$ in these cases, certain qualitative information about these modules is common in all cases.  We summarize the important characteristics in the following

\begin{proposition}\label{prop:decomposition}
If $\Gal(K/F) \simeq G_n$, then $J(K) = \langle \chi \rangle \oplus \bigoplus_{i=0}^n Y_i,$ where
\begin{itemize}
\item $\ell(\chi) = p^{i(K/F)}+1$ for some $i(K/F) \in \{-\infty,0,1,\cdots,n-1\}$, and $e(\chi) \neq 0$; and
\item for each $0 \leq i \leq n$ there exists $\mathfrak{d}_i$ so that $Y_i \simeq \bigoplus_{\mathfrak{d}_i} \F_p[G_i]$, and if $i<n$ then $Y_i \subseteq \ker e$.
\end{itemize}
\end{proposition}

The invariant $i(K/F)$ from this theorem has a variety of interpretations, though there are two that are important for our purposes.  The first has an embedding problem flavor. If $K/F$ embeds in a cyclic extension of degree $p^{n+1}$, then $i(K/F) = -\infty$.  Otherwise, write $K_i$ for the subextension of degree $p^i$ over $F$, and choose $s$ minimally so that $K/K_s$ embeds in a cyclic extension of degree $p^{n-s+1}$; then $i(K/F) = s-1$.  Note in particular that $i(K/F)=-\infty$ whenever $\textrm{char}(K)=p$, since Witt's Theorem on embedding problems in characteristic $p$ (\cite{Wi}) tells us that any $\Z/p^n$-extension embeds in a $\Z/p^{n+1}$-extension in this setting.  The second interpretation of $i(K/F)$ concerns the dimensions of modules generated by elements with nontrivial index: if $e(\gamma) \neq 0$, then $\ell(\gamma) \geq \ell(\chi) = p^{i(K/F)}+1$.

In light of this correspondence, one of the immediate observations to make from Proposition \ref{prop:decomposition} is the following.

\begin{corollary}[cf. {\cite[Prop.~10.2]{EM11}}]
If $K/F$ is a $\Z/p$-extension, then either $\Z/p^2 \twoheadrightarrow \Z/p$ or $M_{p^3} \twoheadrightarrow \Z/p$ is solvable over $K/F$.  More generally, if $K/F$ is a $\Z/p^n$-extension, then for some $i \in \{-\infty,0,1,\cdots,n-1\}$ the embedding problem $A_{p^i+1} \bullet \Z/p^n \twoheadrightarrow \Z/p^n$  is solvable over $K/F$.  
\end{corollary}

\begin{proof}
By Proposition \ref{prop:decomposition} there exists an element $\chi \in J(K)$ so that $e(\chi) \neq 0$ and $\ell(\chi) = p^{i(K/F)}+1$ for some $i(K/F) \in \{-\infty,0,1,\cdots,n-1\}$.  By Proposition \ref{prop:modules.as.galois.groups} this module corresponds to a solution to the embedding problem $A_{p^{i(K/F)}+1} \bullet G_n \twoheadrightarrow G_n$.
\end{proof}

\begin{corollary}
If $K/F$ is a $\Z/p^n$-extension so that for some $j>i$ both $A_i \bullet \Z/p^n \twoheadrightarrow \Z/p^n$ and $A_j \bullet \Z/p^n \twoheadrightarrow \Z/p^n$ are solvable over $K/F$, then $A_{j} \rtimes \Z/p^n \twoheadrightarrow \Z/p^n$ is also solvable over $K/F$.
\end{corollary}

\begin{proof}
Solutions to the embedding problems $A_i \bullet G_n \twoheadrightarrow G_n$ and $A_j \bullet G_n \twoheadrightarrow G_n$ correspond to elements $\gamma_i,\gamma_j \in J(K)$ with nontrivial index and satisfying $\ell(\gamma_i) = i$ and $\ell(\gamma_j) = j$.  By choosing an appropriate $c \in \Z \setminus p\Z$, one has $e(c\gamma_i + \gamma_j) = 0$; furthermore $\ell(c\gamma_i +\gamma_j) =  j$ by the ultrametric property.  Hence $\langle c\gamma_i +\gamma_j \rangle$ corresponds to a solution to $A_j \rtimes G_n \twoheadrightarrow G_n$.
\end{proof}

\begin{remark*}
One cannot make this statement stronger by saying that the appearance of $A_i \bullet G_n$ and $A_j \bullet G_n$ over a field $F$ forces the appearance of $A_j \rtimes G_n$ over $F$ since there are fields $F$ which admit both $\Z/p^2 \simeq A_1 \bullet G_1$- and $M_{p^3} \simeq A_2 \bullet G_1$-extensions, but which do not admit an $H_{p^3} \simeq A_2 \rtimes G_1$-extension.  See, for example, \cite[p.~167]{Br}.
\end{remark*}

\section{Automatic realizations related to $H_{p^3} \Rightarrow M_{p^3}$}

The last two results show that the appearance of certain groups as Galois groups over a field $F$ can force the appearance of other groups as Galois groups over $F$ as well in a non-trivial way.  In this section we will consider other results in this vein.  For a group $G$ and a field $F$, we write $\nu(G,F)$ for the number of extensions $L/F$ with $\Gal(L/F) \simeq G$ in a fixed algebraic closure of $F$.  In the same way, $\nu(G \twoheadrightarrow Q, K/F)$ counts the number of solutions $L/F$ to a given embedding problem $G \twoheadrightarrow Q$ over $K/F$.  

A group $G$ is said to automatically realize a group $H$ if $\nu(G,F)>1$ implies $\nu(H,F)>1$ for any field $F$.  The classic automatic realization theorem is Whaples' result \cite{Wh} that if $p$ is an odd prime number, then $\Z/p^i$ automatically realizes $\Z/p^j$ for all $i<j$.  One can also consider automatic realizations for embedding problems: $G \twoheadrightarrow Q$ is said to automatically realize $H \twoheadrightarrow Q$ if $\nu(G \twoheadrightarrow Q,K/F)\geq 1$ implies $\nu(H \twoheadrightarrow Q,K/F)\geq 1$. 

Brattstr\"om was the first to consider automatic realizations between $H_{p^3}$ and $M_{p^3}$ in \cite{Br}. She showed that $H_{p^3}$ automatically realizes $M_{p^3}$ \cite[Th.~2]{Br}, and that $M_{p^3}$ does not automatically realize $H_{p^3}$ in general \cite[p.~167]{Br}. (The fact that $H_{p^3}$ automatically realizes $M_{p^3}$ was also proved in a different way in \cite[Cor.~12]{EM11}.) However, she does argue that the solvability of the embedding problem $M_{p^3} \twoheadrightarrow G_1$ over a field $K/F$ will imply the solvability of the embedding problem $H_{p^3} \twoheadrightarrow G_1$ over $K/F$ if either $\textrm{char}(K)=p$ or $\xi_p \in N_{K/F}(K^\times)$ \cite[Th.~2]{Br}.  (There are some other known automatic realization results associated with $H_{p^3}$ and $M_{p^3}$.  For instance, in \cite[Th.~1.4A]{J2} it was observed that for any finite group $G$, the group $H_{p^3} \times G$ automatically realizes $M_{p^3} \times G$.  In \cite[Prop.~1.5]{J2} it was also observed that if $$A = \left\langle x,y | x^{p^2} = y^{p^2} = 1, xy = yx^{1+p}\right\rangle,$$ then $M_{p^3}$ automatically realizes $A$.)

By interpreting $H_{p^3}$ as $A_2 \rtimes G_1$ and $M_{p^3}$ as $A_2 \bullet G_1$, we now show Brattstr\"om's results can be viewed from the perspectives of Propositions \ref{prop:modules.as.galois.groups} and \ref{prop:decomposition}.  Indeed, if the embedding problem $A_2 \rtimes G_1 \twoheadrightarrow G_1$ is solvable, then there exists $\langle \gamma \rangle \subseteq J(K)$ with $\ell(\gamma) = 2$ and $e(\gamma) = 0$.  Using the notation of Proposition \ref{prop:decomposition}, if $\ell(\chi) = 2$, then $\langle \chi \rangle$ corresponds to a solution to $A_2 \bullet G_1 \twoheadrightarrow G_1$ and we are done.  Otherwise $\ell(\chi)=1$, and so $\ell(\gamma + \chi) = 2$ by the ultrametric property.  Since $e(\gamma + \chi) = e(\chi) \neq 0$, it therefore follows that $\langle \gamma + \chi \rangle$ corresponds to a solution to $A_2 \bullet G_1 \twoheadrightarrow G_1$ over $K/F$.  

On the other hand, suppose that $L/F$ is an $M_{p^3}$-extension, and let $K/F$ be the unique $\Z/p$-subextension.  If we assume $\xi_p \in N_{K/F}(K^\times)$ then by a famous result of Albert we know that $K/F$ embeds in a $\Z/p^2$-extension, whereas if $\textrm{char}(F)=p$ then it's Witt's theorem which tells us that $K/F$ embeds in a $\Z/p^2$-extension.  In either case we conclude that $\chi$ from Proposition \ref{prop:decomposition} must satisfy $\ell(\chi) = 1$.    Now let $\langle \gamma \rangle$ correspond to the given $M_{p^3}$ extension.  Then $\ell(\gamma) = 2$ and $e(\gamma) \neq 0$.  By choosing an appropriate $c \in \Z \setminus p\Z$, one has $e(\gamma + c \chi) = 0$ and $\ell(\gamma +c\chi) = 2$.  Hence this element corresponds to a solution to $A_2 \rtimes G_1 \twoheadrightarrow G_1$.

Using this same line of reasoning,  we fit this result into a family of similar results which we phrase in the slightly stronger language of automatic realizations of embedding problems.
 
\begin{proposition}\label{prop:automatic.realizations}
We have the following automatic realization results:
\begin{enumerate}
\item\label{it:split.realizes.split} $A_{\ell} \rtimes G_n \twoheadrightarrow G_n$ automatically realizes $A_{\ell+1} \rtimes G_n \twoheadrightarrow G_n$ for $\ell \neq p^k$ with $k \in \{0, 1, \cdots, n-1\}$;
\item\label{it:nonsplit.realizes.split} $A_{\ell} \bullet G_n \twoheadrightarrow G_n$ automatically realizes $A_{\ell} \rtimes G _n \twoheadrightarrow G_n$ for $\ell \neq p^k+1$ with $k \in \{0,1,\cdots,n-1\}$;
\item\label{it:nonsplit.realizes.nonsplit} $A_\ell \bullet G_n \twoheadrightarrow G_n$ automatically realizes $A_{\ell-1} \bullet G_n \twoheadrightarrow G_n$ for $\ell \neq p^k+1$ with $k \in \{0,1,\cdots,n-1\}$; and
\item\label{it:split.realizes.nonsplit} $A_{p^{n-1} + 1} \rtimes G_n \twoheadrightarrow G_n$ automatically realizes $A_{p^{n-1}+k} \bullet G_n \twoheadrightarrow G_n$ for $1\leq k < p^n-p^{n-1}$.
\end{enumerate}
\end{proposition}

\begin{proof}
(\ref{it:split.realizes.split}) was essentially the subject of \cite{MSSauto}, but we reprove the result here.  A solution to the embedding problem $A_\ell \rtimes G_n \twoheadrightarrow G_n$ corresponds to a submodule $\langle \gamma \rangle \subseteq J(K)$ with $\ell(\gamma) = \ell$ and $e(\gamma) = 0$.  By Proposition \ref{prop:decomposition} we can find an $\F_p[G_n]$-basis $\{\chi\} \cup \{\alpha_i\}_{i \in \mathcal{I}}$ for $J(K)$ so that $e(\chi) \neq 0$ and $\ell(\chi)  = p^i(K/F)+1$, and so that for all $i \in \mathcal{I}$ we have $\ell(\alpha_i) = p^{\ell_i}$ for $\ell_i \in \{0,\cdots,n\}$ and $e(\alpha_i) = 0$ when $\ell(\alpha_i)<p^n$.  Express
$$\gamma = f \chi + \sum_{i \in \mathcal{I}} f_i \alpha_i$$
with $f,f_i \in \F_p[G_n]$; since $e(\gamma) = 0$, it must be the case that $f \in \langle \sigma-1 \rangle$; likewise since $\ell(\gamma) = \ell < p^n$ we must have $f_i \in \langle \sigma-1 \rangle$ for all $i \in \mathcal{I}$ such that $\ell(\alpha_i) = p^n$.  Now we have $\ell(\gamma) = \max\left\{\ell(f\chi),\left\{\ell(f_i\alpha_i): i \in \mathcal{I} \right\} \right\}$ using the ultrametric property together with the $\F_p[G_n]$-independence of the set $\{\chi\} \cup \{\alpha_i\}_{i \in \mathcal{I}}$.  

We consider two cases.  If $\ell(\gamma) = \ell(f \chi)$, then since $\ell \neq p^k$ for any $k \in \{0,1,\cdots,n-1\}$ and $\ell(f \chi) = p^{i(K/F)}+1-v(f)$, we must have $v(f) \geq 2$.  Hence $\ell((\sigma-1)^{v(f)-1}\chi) = \ell+1$ and $e((\sigma-1)^{v(f)-1}\chi) = 0$.  Hence $A_{\ell+1} \rtimes G_n \twoheadrightarrow G_n$ has a solution.

On the other hand, if $\ell(\gamma) = \ell(f_i\alpha_i)$ for some $i \in \mathcal{I}$, then since $\ell(\alpha_i) = p^{\ell_i}$, it must be the case that $v(f_i) \geq 1$.  But then $\ell((\sigma-1)^{v(f_i)-1}\alpha_i) = \ell+1$ and $e((\sigma-1)^{v(f_i)-1}\alpha_i)=0$.  Again, we have a solution to $A_{\ell+1} \rtimes G_n \twoheadrightarrow G_n$.

(\ref{it:nonsplit.realizes.split}) has two potential proofs.  From the group-theoretic perspective, $A_{\ell-1} \rtimes G_n$ is a quotient of $A_\ell \bullet G_n$, and hence $A_\ell \bullet G_n \twoheadrightarrow G_n$ trivially automatically realizes $A_{\ell -1} \rtimes G_n\twoheadrightarrow G_n$.  Then (\ref{it:split.realizes.split}) tells us that $A_{\ell-1} \rtimes G_n \twoheadrightarrow G_n$ automatically realizes $A_\ell \rtimes G_n \twoheadrightarrow G_n$.  Alternatively, one could prove this result module-theoretically.  In this case, a solution to $A_{\ell} \bullet G_n \twoheadrightarrow G_n$ implies the existence of a submodule $\langle \gamma \rangle \subseteq J(K)$ with $\ell(\gamma) = \ell$ and $e(\gamma) \neq 0$.  Choose an appropriate value $c \in \Z \setminus p\Z$ so that $e(c \chi+ \gamma) = 0$, and note that $\ell(c\chi + \gamma) = \ell(\gamma)$ by the ultrametric property (recall that if $e(\gamma) \neq 0$ then $\ell(\gamma) \geq \ell(\chi)$, and our hypothesis forces this inequality to be strict).  Hence $\langle c\chi +\gamma \rangle$ corresponds to a solution to $A_{\ell} \rtimes G_n \twoheadrightarrow G_n$.

For the proof of (\ref{it:nonsplit.realizes.nonsplit}), suppose that $\langle \gamma \rangle \subseteq J(K)$ corresponds to a solution to the embedding problem $A_\ell \bullet G_n \twoheadrightarrow G_n$.  The module-theoretic proof of (\ref{it:nonsplit.realizes.split}) gives us a module $\langle c\chi +\gamma \rangle$ with $e(c\chi +\gamma) = 0$ and $\ell = \ell(c\chi+ \gamma) > \ell(\chi)$.  If $\ell(\chi) = \ell-1$, then $\langle \chi \rangle$ corresponds to a solution to the embedding problem $A_{\ell-1} \bullet G_n \twoheadrightarrow G_n$. Otherwise $\ell(\chi + (\sigma-1)(c\chi+ \gamma)) = \ell-1$, and of course this module is generated by an element of nontrivial index.  Hence it corresponds to a solution to the embedding problem $A_{\ell-1} \bullet G_n \twoheadrightarrow G_n$.

Finally, to prove (\ref{it:split.realizes.nonsplit}), we use (\ref{it:split.realizes.split}) to conclude that there is a solution to $A_{p^{n-1}+k}\rtimes G_n \twoheadrightarrow G_n$ over $K/F$.  Write $\langle \gamma \rangle \subseteq J(K)$ for the module that corresponds to a solution to the embedding problem $A_{p^{n-1}+k} \rtimes G_n \twoheadrightarrow G_n$.  Obviously $e(\chi + \gamma) \neq 0$, and if we can show $\ell(\chi + \gamma) = p^{n-1}+k$ then this module will correspond to a solution to the embedding problem $A_{p^{n-1}+k}\bullet G_n \twoheadrightarrow G_n$.  We know that $\ell(\chi+\gamma) \leq \max\{p^{n-1}+k,p^{i(K/F)}+1\} = p^{n-1}+k$, with equality if either $k>1$ or $i(K/F) \neq n-1$.  Hence $\ell(\chi+\gamma) < p^{n-1}+k$ implies both $k=1$ and $i(K/F) = n-1$.  But since $e(\chi+\gamma) \neq 0$ and $\chi$ is an element of minimal length amongst elements of non-trivial index, we must also have $p^{n-1}+1>\ell(\chi+\gamma) \geq \ell(\chi) = p^{n-1}+1$, a contradiction.  
\end{proof}

Finally, we give a proposition which builds on Proposition \ref{prop:automatic.realizations}(\ref{it:split.realizes.nonsplit}), but doesn't require the underlying $\F_p[G_n]$-module to have such a large dimension.  Recall that $K_i$ denotes the intermediate field in the extension $K/F$ such that $[K_i:F]=p^i$.  

\begin{proposition}
Let $i \in \{0,1,\cdots,n-1\}$ be given.  If the embedding problems $G_{n-i} \twoheadrightarrow G_{n-i-1}$ over $K/K_{i+1}$ and $A_{p^i+1}\rtimes G_n \twoheadrightarrow G_n$ over $K/F$ are both solvable, then the embedding problem $A_{p^i+k}\bullet G_n \twoheadrightarrow G_n$ is also solvable over $K/F$ for all $1 \leq k \leq p^{i+1}-p^i$.
\end{proposition}

\begin{proof}
The proof of this result is essentially the same as the proof of Proposition \ref{prop:automatic.realizations}(\ref{it:split.realizes.nonsplit}), though our additional hypothesis concerning the solvability of $G_{n-i}\twoheadrightarrow G_{n-i-1}$ over $K/K_{i+1}$ tells us that $i(E/F)\leq i$, so that $\ell(\chi) \leq p^i+1$. To find a solution to the desired embedding problem, we note that the solvability of $A_{p^i+1}\rtimes G_n \twoheadrightarrow G_n$ over $K/F$ implies the solvability of $A_{p^i+k}\rtimes G_n \twoheadrightarrow G_n$ over $K/F$; let $\langle \gamma \rangle$ be a module in $J(K)$ which corresponds to a solution to this embedding problem.  Then the module $\langle \gamma + \chi \rangle$ will be a solution to the embedding problem $A_{p^i+k} \bullet G_n \twoheadrightarrow G_n$.
\end{proof}

\section{Enumerating Galois extensions related to $H_{p^3}$ and $M_{p^3}$}

We now shift focus and concentrate on enumeration results related to $H_{p^3}$ and $M_{p^3}$, particularly within the family of groups $A_i \rtimes G_1$ and $A_i \bullet G_1$.  One of the results already known in this vein comes from a paper by Brattstr\"om (\cite[Thm.~5]{Br}) where it is shown that $$\nu(M_{p^3},F) = (p^2-1)\nu(H_{p^3},F)  \quad \mbox{ if }\xi_{p^2} \in F \mbox{ or }\textrm{char}(F) = p.$$  Here we present a stronger result that drops the assumption that $\xi_{p^2} \in F$ when $\textrm{char}(F) \neq p$ and gives a closed formula for the difference $\nu(M_{p^3},F)-(p^2-1)\nu(H_{p^3},F)$.  

Before arriving at this result, we will first need to consider the following $\F_p$-subspace of $J(F)$: $$\mathfrak{N} = \left\{\begin{array}{ll}
\frac{N_{F(\xi_{p^2})/F}(F(\xi_{p^2})^\times)~F^{\times p}}{F^{\times p}}, &\mbox{ when } \xi_p \in F\\[10pt]
J(F),&\mbox{ when } \textrm{char}(F) = p\\[10pt]
\mathcal{T}\left(\frac{N_{F(\xi_{p^2})/F(\xi_p)}(F(\xi_{p^2})^\times)~F(\xi_p)^{\times p}}{F(\xi_p)^{\times p}}\right), &\mbox{ when }\textrm{char}(F) \neq p \mbox{ and } \xi_p \not\in F.
\end{array}\right.$$  Recall that in the latter case we write $\tau$ for the generator of $\Gal(F(\xi_p)/F)$, and that $\mathcal{T} \in \Z[\langle \tau\rangle]$. Because $\gcd(|\tau|,p)=1$, we have that $$(N_{F(\xi_{p^2})/F(\xi_p)}(\alpha))^\mathcal{T} = N_{F(\xi_{p^2})/F(\xi_p)}(\alpha^{\mathcal{T}})$$ for any $\alpha \in F(\xi_{p^2}^\times)$.  It follows that $\mathfrak{N}$ is the $t$-eigenspace for $\tau$ within $\frac{N_{F(\xi_{p^2})/F(\xi_p)}(F(\xi_{p^2})^\times)~F(\xi_p)^{\times p}}{F(\xi_p)^{\times p}}$.  The importance of this space is that it parameterizes those elements of $J(F)$ whose corresponding $G_1$-extensions admit a solution to the embedding problem $G_2 \twoheadrightarrow G_1$.

\begin{proposition}\label{prop:parameterizing.embeddable.extensions}
Suppose that $f \in J(F)$ corresponds to the $\Z/p$-extension $K/F$.  Then the embedding problem $G_2 \twoheadrightarrow G_1$ has a solution over $K/F$ if and only if $f \in \mathfrak{N}$.
\end{proposition}

\begin{proof}
This result is \cite[Thm.~1]{BP} if $\xi_p \in F$; when $\xi_p \not\in F$ and $\ch{F} \neq p$, the result follows by descent.  If $\ch{F}=p$ then the embedding problem $G_2 \twoheadrightarrow G_1$ is always solvable.  
\end{proof}

We are now prepared to give a generalization of Brattstr\"{o}m's result connecting $\nu(H_{p^3},F)$ and $\nu(M_{p^3},F)$.  In the statement of this theorem, we use $\binom{n}{m}_p$ for the $p$-binomial coefficient which counts the number of $m$-dimensional subspaces within an ambient $n$-dimensional $\F_p$-space.  It is a nice exercise in linear algebra to show that $$\binom{n}{m}_p = \frac{(p^n-1)\cdots (p^{n-m+1}-1)}{(p^m-1)\cdots (p-1)}$$ (see, e.g., \cite[Ch.~7]{KC}).  It is also worth remarking that by changing the prime $p$ to a variable $q$, one gets the quantum binomial coefficient ${n \brack m}_q$ introduced by Gauss.  Observe also that $\binom{n}{1}_p = \frac{p^n-1}{p-1}$.

\begin{theorem}\label{thm:generalizing.brattstrom} Let $p$ be an odd prime, and let $\mathfrak{N}$ be the subspace of $J(F)$ defined above.  Then
$$\nu(M_{p^3},F) = (p^2-1) \nu(H_{p^3},F) + \left(\binom{\dim_{\F_p}J(F)}{1}_p - \binom{\dim_{\F_p}\mathfrak{N}}{1}_p\right)\frac{|J(F)|}{p^2}.$$
\end{theorem}

Before proving this result, we observe that when $\ch{F}\neq p$ then $|J(F)|<p^2$ is only possible when $p=2$ and $n=1$. Since we are focusing on the case $p>2$, the term $|J(F)|/p^2$ is therefore an integer in this case.  On the other hand, when $\ch{F}=p$ then $|J(F)|<p^2$ is possible, but in this case the term $\binom{\dim J(F)}{1}_p-\binom{\dim \mathfrak{N}}{1}_p$ vanishes.

\begin{proof}
Suppose that $K/F$ is a $G_1$-extension.  From Proposition \ref{prop:modules.as.galois.groups}, we know that there is a bijection $$\left\{\begin{array}{c}\F_p[G_1]-\mbox{submodules }\\M \subseteq J(K) \mbox{ with }\\M \simeq A_2\mbox{ and }M \subseteq \ker(e)\end{array}\right\} \leftrightarrow \left\{\begin{array}{c}\mbox{Solutions to}\\H_{p^3} \twoheadrightarrow G_1\\\mbox{over }K/F\end{array}\right\}.$$  Now since any given module $M \simeq A_2$ satisfies $|\{m \in M: \ell(m)=2\}|=p^2-p$, and since any element $m \in J(K) \cap \ker(e)$ with $\ell(m)=2$ generates a submodule corresponding to a solution to $H_{p^3} \twoheadrightarrow G_1$, we get 
$$\nu(H_{p^3} \twoheadrightarrow G_1,K/F) = \frac{1}{p^2-p} \left|\{\gamma \in J(K)\cap\ker(e): \ell(\gamma)=2\}\right|.$$  Since any $H_{p^3}$ extension has $p+1$ quotients isomorphic to $\Z/p$, we therefore conclude 
$$\nu(H_{p^3},F) = \frac{1}{p+1} \sum_{K/F} \frac{1}{p^2-p} \left|\{\gamma \in J(K) \cap \ker(e): \ell(\gamma)=2\}\right|.$$

Now we will consider $M_{p^3}$ extensions, so again let $K/F$ be a given $G_1$-extension.  We will fix an element $\chi$ as in Proposition \ref{prop:decomposition}.  According to Proposition \ref{prop:modules.as.galois.groups}, we know that there is a bijection $$\left\{\begin{array}{c}\F_p[G_1]-\mbox{submodules }\\M \subseteq J(K) \mbox{ with }\\M \simeq A_2\mbox{ and }M \not\subseteq \ker(e)\end{array}\right\} \leftrightarrow \left\{\begin{array}{c}\mbox{Solutions to}\\M_{p^3} \twoheadrightarrow G_1\\\mbox{over }K/F\end{array}\right\}.$$ Again, any module $M \simeq A_2$ satisfies $|\{m\in M: \ell(m)=2\}| = p^2-p$, and if the module satisfies $M \not\subseteq \ker(e)$ then it must be that any element from $\{m \in M: \ell(m)=2\}$ has $e(m)\neq 0$.  (One can see this in several ways, but here's an embedding problem argument: if there were an element with $\ell(m)=2$ and $e(m)=0$, then $M=\langle m\rangle$ would solve $H_{p^3} \twoheadrightarrow G_1$ instead of $M_{p^3} \twoheadrightarrow G_1$.) Hence we conclude that 
$$\nu(M_{p^3} \twoheadrightarrow G_1,K/F) = \frac{1}{p^2-p} \left|\{\alpha \in J(K): \alpha\not\in \ker(e), \ell(\alpha)=2\}\right|.$$ We claim that $\{\alpha \in J(K): \alpha\not\in \ker(e), \ell(\alpha)=2\}$ is equal to
\begin{equation}\label{eq:expressing.nontrivial.elements.of.length.2}\left\{\begin{array}{ll}\displaystyle\bigcup_{c=1}^{p-1} c\chi+ \{\gamma \in J(K) \cap \ker(e): \ell(\gamma)=2\},&\mbox{ if }\ell(\chi)=1\\[20pt]\displaystyle\bigcup_{c=1}^{p-1} c\chi+ \{\gamma \in J(K) \cap \ker(e): \ell(\gamma)=2\} \cup \bigcup_{c=1}^{p-1}c\chi+ \left( J(K)^G \cap \ker(e)\right),&\mbox{ if }\ell(\chi)=2.\end{array}\right.
\end{equation} It will be convenient to translate these two conditions into equivalent statements: in the language of Proposition \ref{prop:decomposition}, the condition $\ell(\chi)=1$ is equivalent to $i(K/F) = -\infty$, whereas in the language of embedding problem this condition says that $G_2 \twoheadrightarrow G_1$ is solvable over $K/F$.  On the other hand, $\ell(\chi)=2$ translates to $i(K/F)=0$ in the language of Proposition \ref{prop:decomposition}, or to the embedding problem statement that $G_2 \twoheadrightarrow G_1$ does not have a solution over $K/F$.

If we assume the equality of sets from equation (\ref{eq:expressing.nontrivial.elements.of.length.2}) for the time being, then since any given $M_{p^3}$-extension has a unique $\Z/p$-quotient, it will follow that
\begin{align*}
\nu(M_{p^3},F) &= \sum_{K/F} \nu(M_{p^3} \twoheadrightarrow G_1,K/F)\\
&=\sum_{K/F} \frac{1}{p^2-p} |\{\alpha \in J(K): \alpha \not\in\ker(e), \ell(\alpha)=2\}|\\
&=\frac{p-1}{p^2-p}\left(\sum_{K/F} |\{\gamma \in J(K)\cap \ker(e): \ell(\gamma)=2\}| +  \sum_{i(K/F)=0} |J(K)^G \cap \ker(e)|\right)\\
&=(p^2-1) \nu(H_{p^3},F) + \frac{1}{p} \sum_{i(K/F)=0} |J(K)^G \cap \ker(e)|.
\end{align*} 
We connect this expression to the desired formula for $\nu(M_{p^3},F)$ by considering two cases.  First, suppose $\ch{F}=p$.  In this case we have $i(K/F)=-\infty$ for any $G_1$-extension $K/F$, and hence the latter sum is empty.  But in this case note that the second summand from the desired formula for $\nu(M_{p^3},F)$ also vanishes since $J(F)=\mathfrak{N}$ in this case.  Hence we have the desired result in this case.  (Note: when $\xi_{p^2}\in F$ we also have $\mathfrak{N} = J(F)$ so that the latter term vanishes; this proves the other case of Brattstr\"{o}m's result.)

Now suppose that $\ch{F}\neq p$.  If $K/F$ is a $\Z/p$-extension, let $\iota: J(F) \to J(K)$ be the map induced by the natural inclusion; from \cite[Lemma 8]{MSS1} we have that $J(K)^G \cap \ker(e) = \iota(J(F))$, and from Kummer Theory we have that $|\iota(J(F))| = |J(F)|/p$.  Hence we can continue our chain of equalities by writing
$$\nu(M_{p^3},F) = (p^2-1) \nu(H_{p^3},F) + \frac{|J(F)|}{p^2} \left|\{K/F: G_2 \twoheadrightarrow G_1 \mbox{ is not solvable}\}\right|.$$  The only thing left to argue, then, is that the number of extensions $K/F$ for which $G_2 \twoheadrightarrow G_1$ is not solvable is given by $\binom{\dim J(F)}{1}_p-\binom{\dim \mathfrak{N}}{1}_p$, though this is straightforward given Proposition \ref{prop:parameterizing.embeddable.extensions}: $\binom{\dim J(F)}{1}_p$ counts all $\Z/p$-extensions, and $\binom{\dim \mathfrak{N}}{1}_p$ counts all $\Z/p$-extensions which  admit a solution to the embedding problem $G_2 \twoheadrightarrow G_1$.

To finish the proof, then, we simply need to show that $\{\alpha \in J(K): \alpha \not\in \ker(e), \ell(\alpha) = 2\}$ is equal to the expression from equation (\ref{eq:expressing.nontrivial.elements.of.length.2}).  
Suppose first that $\ell(\chi)=1$.  Let $\alpha \in J(K)$ be given which satisfies $e(\alpha) \neq 0$ and $\ell(\alpha)=2$. Then there exists some $c \in \F_p^\times$ so that $\alpha -c\chi \in \ker(e)$, and the ultrametric property gives $\ell(\alpha -c\chi) = 2$. This gives one containment.  For the other, if $\gamma \in J(K) \cap \ker(e)$ and $\ell(\gamma)=2$, then the ultrametric property tells us that for any $c \in \F_p^\times$ we have $\ell(c\chi+ \gamma)=2$, and of course $e(c\chi+ \gamma) \neq 0$.

Now suppose that $\ell(\chi)=2$.  Let $\alpha \in J(K)$ be given which satisfies $e(\alpha) \neq 0$ and $\ell(\alpha)=2$. Then there exists some $c \in \F_p^\times$ so that $\alpha -c \chi \in \ker(e)$, though this time the ultrametric property only gives $\ell(\alpha -c\chi) \leq 2$. Hence either $\alpha -c\chi$ has length $2$ or $\alpha -c\chi \in J(K)^G$.  This proves one containment.  For the other, the ultrametric property makes it clear that if $\gamma \in J(K)^G$ and $e(\gamma)=0$, then for any $c \in \F_p^\times$ we have that $e(c\chi+\gamma)\neq 0$ and $\ell(c\chi+\gamma)=2$.  We claim that for $\gamma \in \ker(e)$ with $\ell(\gamma) = 2$, it is still true that for all $c \in \F_p^\times$ we have $e(c\chi+\gamma) \neq 0$ and $\ell(c\chi+\gamma)=2$.  The first statement is clear; for the second, note that if $\ell(c\chi+\gamma) <2$ then this implies that $c\chi+ \gamma$ is an element of nontrivial index with length $1$, a contradiction to the fact that $\chi$ is an element of nontrivial index with minimal length.
\end{proof}

Notice that in the previous theorem we had to be careful in using our methodology to enumerate extensions since our modules naturally parameterize solutions to embedding problems over a given $\Z/p$-extension $K/F$.  In the case of $H_{p^3}$ extensions of $F$, we had to account for the fact that a given $H_{p^3}$ extension of $F$ solves embedding problems over $p+1$ distinct $\Z/p$-extension of $F$.  To adopt the methodology of the previous theorem to a broader class of groups, we will be interested in determining when groups of the form $A_i \rtimes G_n$ or $A_i \bullet G_n$ have precisely one normal subgroup which is elementary $p$-abelian and whose quotient is isomorphic $\Z/p^n$.  

\begin{lemma}\label{le:different.extensions.for.long.modules}
Suppose that $\Gal(L/F) \simeq A_i \rtimes G_n$ for $i \geq p^{n-1}+2$.  Then there is a unique intermediate $\Z/p^n$-extension $K/F$ so that $L/K$ is elementary $p$-abelian.  Likewise if $\Gal(L/F) \simeq A_i \bullet G_n$ for $i \geq p^{n-1}+1$, then there is a unique intermediate $\Z/p^n$-extension $K/F$ so that $L/K$ is elementary $p$-abelian.  
\end{lemma}

\begin{proof}In each case, it is obvious that $T=\{(f,1):f \in A_i\}$ is a normal subgroup which is elementary $p$-abelian.  Equally clear is that the fixed field is a $\Z/p^n$-extension.  Before proceeding, we observe that the collection of elements in $A_i \rtimes G_n$ with order $p$ are those of the form $(f,\sigma^j)$ where $p^{n-1} \mid j$.

For the sake of contradiction, suppose that $H \neq T$ is an elementary $p$-abelian normal subgroup of $A_i \rtimes G_n$ with quotient $\Z/p^n$; this forces $|H| = p^i$ and $H \setminus T \neq \emptyset$.  Choose $(f,\sigma^j) \in H \setminus T$ so that $\ell(f)$ is maximal amongst elements within $H\setminus T$.  After taking a suitable power of $(f,\sigma^j)$ if necessary, we can assume that our element is $(f,\sigma^{p^{n-1}})$.  (Note that $(f,\sigma^j)^t = \left(\sum_{i=0}^{t-1} \sigma^{it} f,\sigma^{jt}\right)$, and that $\sum_{k=0}^{t-1} \sigma^{jk}$ is a unit in $\F_p[G]$ for $1\leq t \leq p-1$, so that $\ell(f) = \ell\left(\sum_{k=0}^{t-1} \sigma^{jk} f\right)$.)

Since $H$ is normal we know that the commutator $[(0,\sigma),(f,\sigma^j)] = ((\sigma-1) f,1)$ is an element of $H$; repeating this procedure shows  $\{(g,1): \ell(g)<\ell(f)\} \subseteq H$.  By the maximality of $(f,\sigma^{p^{n-1}})$, there exists no $(g,\sigma^k) \in H$ with $k \neq p^n$ and $\ell(g)>\ell(f)$.  We argue that there are also no elements $(g,1) \in H$ with $\ell(g)>\ell(f)$; otherwise $(h,1)(f,\sigma^{p^{n-1}}) = (h+f,\sigma^{p^{n-1}}) \in H$, and the ultrametric property gives $\ell(h+f)=\ell(h)$, violating the maximality condition defining $f$.  Notice also that since $((\sigma-1)f,1) \in H$ and $H$ is assumed abelian, we must have
$$(f + (\sigma-1)f,\sigma^{p^{n-1}}) = ((\sigma-1)f,1)(f,\sigma^{p^{n-1}}) = (f,\sigma^{p^{n-1}})((\sigma-1)f,1) = (f + \sigma^{p^{n-1}}(\sigma-1)f,\sigma^{p^{n-1}}).$$ It therefore follows that $(\sigma-1)f$ is fixed by $\sigma^{p^{n-1}}$, and so $\ell(f) \leq p^{n-1}+1$.

Now suppose that $H$ contains no elements of the form $(g,\sigma^k)$ with $k \neq p^n$ and $\ell(g)<\ell(f)$.  Then we have
$$H = \left \langle (f,\sigma^{p^{n-1}}),((\sigma-1)f,1),((\sigma-1)^2f,1),\cdots,((\sigma-1)^{\ell(f)-1}f,1)\right\rangle_{\F_p},$$ and hence $|H| = p^{\ell(f)} \leq p^{p^{n-1}+1} < p^i$, a contradiction.

On the other hand, suppose that $H$ does contain an element $(g,\sigma^k)$ with $k \neq p^n$ and $\ell(g)<\ell(f)$.  One can argue that $(g,\sigma^k)$ can be chosen to take the form $(g,\sigma^{p^{n-1}})$ as before, and since we have already established $(g,1) \in H$, we conclude that $(0,\sigma^{p^{n-1}}) \in H$. This gives $(f,\sigma^{p^{n-1}})(0,\sigma^{p^{n-1}})^{-1} = (f,1) \in H$ as well.  The commutivity of $H$ implies $$ (f,\sigma^{p^{n-1}}) = (f,1)(0,\sigma^{p^{n-1}}) = (0,\sigma^{p^{n-1}})(f,1) = (\sigma^{p^{n-1}} f,\sigma^j).$$ We conclude that $f$ is fixed by the action of $\sigma^{p^{n-1}}$, so that in fact $\ell(f) \leq p^{n-1}$ in this case.  But then $$H = \left \langle(f,1),((\sigma-1)f,1),\cdots,((\sigma-1)^{\ell(f)-1}f,1), (0,\sigma^{p^{n-1}})\right\rangle_{\F_p},$$ contradicting the fact that $i \geq p^{n-1}+2$ in this case.  This completes the proof for $A_i \rtimes G_n$.


Now suppose that $H \neq T$ is an elementary $p$-abelian normal subgroup of $A_i \bullet G_n$ with quotient $\Z/p^n$; this forces $|H| = p^i$.  Again, we will choose $(f,\sigma^j) \in H \setminus T$ so that $\ell(f)$ is maximal amongst elements within $H \setminus T$, and we again observe that we can assume that this element is of the form $(f,\sigma^{p^{n-1}})$.  We begin our argument in this case by claiming that $\ell(f) > p^n-p^{n-1}$ is necessary.  To do so, we first establish some notation.  For an integer $a$, we write $\overline{a}$ for the least non-negative residue of $a$ modulo $p^n$.  For $m \in \N$, we then define $c_j(m)$ inductively: $c_j(1)= 0$, and $$c_{j}(m+1) = \left\{ \begin{array}{ll}c_j(m)&\mbox{ if }\overline{mj}+j<p^n\\c_j(m)+1&\mbox{ if }\overline{mj}+j \geq p^n.\end{array}\right.$$  Then one can use induction to show that $$(f,\sigma^{p^{n-1}})^m = \left(\sum_{k=0}^{m-1} \sigma^{kp^{n-1}} f + c_{p^{n-1}}(m) (\sigma-1)^{i-1},\sigma^{mp^{n-1}}\right).$$ Take $m=p$ and observe that $1 \leq c_{p^{n-1}}(p)<p$ (the first inequality follows because $p^{n-1}p \geq p^n$, and the latter because $c_{p^{n-1}}(1)=0$ and $c_{p^{n-1}}(m+1)-c_{p^{n-1}}(m) \leq 1$ for all $m$).  If $\ell(f)\leq p^{n-1}(p-1)$ this means that \begin{equation*}
\begin{split}(f,\sigma^{p^{n-1}})^p &= \left(\sum_{k=0}^{p-1} \sigma^{k{p^{n-1}}}f+c_{p^{n-1}}(p)(\sigma-1)^{i-1},1\right) \\&= \left((\sigma^{p^{n-1}}-1)^{p-1}f + c_{p^{n-1}}(p)(\sigma-1)^{i-1},1\right) = (c_{p^{n-1}}(p)(\sigma-1)^{i-1},1) \neq (0,1),\end{split}\end{equation*} contrary to the assumption that $H$ is elementary $p$-abelian.

Now observe that $[(0,\sigma),(f,\sigma^{p^{n-1}})] \in H$, and hence should commute with $(f,\sigma^{p^{n-1}})$.  But a computation reveals
$$[(f,\sigma^{p^{n-1}}),[(0,\sigma),(f,\sigma^{p^{n-1}})]] = ((\sigma-1)(\sigma^{p^{n-1}}-1)f,1).$$ Provided either $p>3$ or $n=1$, this element is nontrivial because $\ell(f) \geq p^n-p^{n-1} > p^{n-1}+1$, and so $H$ is not commutative. 

We will handle the remaining case $p=3$ and $n=1$ directly.  We have the group $A_i \bullet \Z/3\Z$, where $2 \leq i \leq 3$, and we want to show this group has a unique $\Z/p$-quotient whose corresponding normal subgroup is $\Z/p \times \Z/p$.  When $i=2$ then the group is simply $M_{3^3}$, and we already know the result for this group.  When $i=3$ then $A_i \bullet \Z/3 \simeq A_i \rtimes \Z/3$, and we have already established the desired result from the first part of this theorem.
\end{proof}

\begin{theorem}
For an extension $K/F \simeq \Z/p^n$ and $p^{n-1}+2 \leq i < p^n$, $$\nu(A_i \bullet G_n \twoheadrightarrow G_n,K/F) = (p-1)\nu(A_i \rtimes G_n \twoheadrightarrow G_n,K/F).$$
Moreover, for a field $F$ and $p^{n-1}+2 \leq i < p^n$, $$\nu(A_i \bullet G_n,F) = (p-1) \nu(A_i \rtimes G_n,F).$$
\end{theorem}

\begin{proof}
We begin by noting that Lemma \ref{le:different.extensions.for.long.modules} tells us that the second statement follows from the first, since a given $A_i \rtimes G_n$ or $A_i \bullet G_n$ extension of $F$ has a unique $\Z/p^n$-subextension $K/F$ such that $L/K$ is elementary $p$-abelian.  Hence such an extension is parameterized uniquely as a module within $J(K)$, and so we have
$$\nu(A_i \rtimes G_n,F) = \sum_{\Gal(K/F)\simeq G_n} \nu(A_i \rtimes G_n \twoheadrightarrow G_n,K/F)$$ and likewise for the non-semidirect product.  Hence, we will focus on proving the first result.

To do so, we note that using Proposition \ref{prop:modules.as.galois.groups} and that fact that any cyclic module of dimension $i$ has $p^i-p^{i-1}$ many generators, one has
\begin{align*}
\nu(A_i\rtimes G_n\twoheadrightarrow G_n,K/F) &= \frac{1}{p^i-p^{i-1}}\left|\{\gamma \in J(K): \gamma \in \ker(e), \ell(\gamma)=i\}\right|\\
\nu(A_i\bullet G_n\twoheadrightarrow G_n,K/F) &= \frac{1}{p^i-p^{i-1}}\left|\{\alpha \in J(K): \alpha \not\in \ker(e), \ell(\alpha)=i\}\right|.
\end{align*}

The desired result will follow if we can show that, for the element $\chi \in J(K)$ from Proposition \ref{prop:decomposition}, one has
$$\{\alpha \in J(K):\alpha \not\in\ker(e), \ell(\alpha)=i\} = \bigcup_{c=1}^{p-1} c \chi + \{\gamma \in J(K): \gamma \in \ker(e),\ell(\gamma)=i\}.$$  For this, note that $\ell(\chi)=p^{i(E/F)}+1 < i$, and so if $\ell(\gamma)= i$ then the ultrametric property gives $\ell(c\chi + \gamma)=i$.  Of course if $\gamma\in \ker(e)$ as well, then $e(c\chi+ \gamma) \neq 0$.  This gives one containment.  For the other, note that if $\alpha \not\in\ker(e)$ then there exists some $c \in \F_p^\times$ so that $\alpha-c\chi \in \ker(e)$; when $\ell(\alpha)=i$, the ultrametric property again gives $\ell(\alpha-c\chi)=i$, and hence $\alpha =c\chi + \alpha -c \chi \in c\chi + \{\gamma \in J(K): \gamma \in \ker(e),\ell(\gamma)=i\}$.
\end{proof}

If one is willing to settle for counting only solutions to embedding problems, one can extend these same ideas to express $\nu(A_\ell \rtimes G_1 \twoheadrightarrow G_1,K/F)$ in terms of $\nu(H_{p^3} \twoheadrightarrow G_1,K/F)$ and $\nu(\Z/p \times \Z/p \twoheadrightarrow \Z/p,K/F)$.

\begin{theorem}
For $2 \leq \ell \leq p-1$ and a $\Z/p$-extension $K/F$, we have
\begin{align*}
\nu(A_\ell \rtimes \Z/p &\twoheadrightarrow\Z/p,K/F) \\&= \nu(H_{p^3} \twoheadrightarrow \Z/p,K/F) \left(\frac{1}{p}+\frac{(p-1)\nu(H_{p^3}\twoheadrightarrow \Z/p,K/F)}{1+(p-1)\nu(\Z/p\times\Z/p\twoheadrightarrow\Z/p,K/F)}\right)^{\ell-2}.
\end{align*}
\end{theorem}

\begin{proof}
We know that
$$\nu(H_{p^3} \twoheadrightarrow \Z/p,K/F) = \frac{1}{p^2-p} \left|\left\{\gamma \in J(K): \ell(\gamma) = 2 \mbox{ and } \gamma \in \ker(e)\right\}\right|.$$  Likewise we have
$$\nu(A_\ell \rtimes \Z/p \twoheadrightarrow \Z/p,K/F) = \frac{1}{p^\ell-p^{\ell-1}} \left|\left\{\gamma \in J(K): \ell(\gamma) = \ell \mbox{ and } \gamma \in \ker(e)\right\}\right|.$$  

By Proposition \ref{prop:decomposition} we know that the $\F_p[G_1]$-structure of $\ker(e)$ is
$$\ker(e) \simeq \bigoplus^{\mathfrak{d}_0} \F_p \oplus \bigoplus^{\mathfrak{d}_1}\F_p[G_1]/(\sigma-1)^{p-2}.$$ Hence it is relatively simply to see that for $1\leq \ell \leq p-1$ we have
\begin{align*}
\left|\left\{\gamma \in J: \ell(\gamma) \leq \ell \mbox{ and }\gamma \in \ker(e)\right\}\right| &= p^{\mathfrak{d}_0+\ell\mathfrak{d}_1}.\\
\end{align*}
Hence for $2 \leq \ell \leq p-1$ we have
$$\left|\left\{\gamma \in J: \ell(\gamma) = \ell \mbox{ and }\gamma \in \ker(e)\right\}\right| = p^{\mathfrak{d}_0+\ell \mathfrak{d}_1} - p^{\mathfrak{d}_0+(\ell-1)\mathfrak{d}_1}= p^{\mathfrak{d}_0+(\ell-1)\mathfrak{d}_1}\left(p^{ \mathfrak{d}_1}-1\right)$$ and for $\ell=1$ we get
$$\left|\left\{\gamma \in J: \ell(\gamma) = 1 \mbox{ and }\gamma \in \ker(e)\right\}\right| = p^{\mathfrak{d}_0+ \mathfrak{d}_1} - 1.$$
Hence for $2 \leq \ell \leq p-1$, one calculates
\begin{align}
\nu(A_\ell \rtimes \Z/p \twoheadrightarrow \Z/p,K/F) &= \frac{1}{p^\ell-p^{\ell-1}}\left|\{\gamma \in \ker(e): \ell(\gamma) = \ell\}\right|\notag
\\
&=\frac{1}{p^{\ell-1}(p-1)}p^{\mathfrak{d}_0+(\ell-1)\mathfrak{d}_1}\left(p^{ \mathfrak{d}_1}-1\right)\notag
\\\label{eq:realization.count.rewritten}
&=p^{\mathfrak{d}_0+\mathfrak{d}_1-1}\frac{p^{\mathfrak{d}_1}-1}{p-1}\left(p^{\mathfrak{d}_1-1}\right)^{\ell-2}.
\end{align}
Of course the case $\ell=1$ follows in a similar way:
\begin{align}\label{eq:count.on.Zp.times.Zp}
\nu(\Z/p \times \Z/p \twoheadrightarrow \Z/p,K/F) = \frac{p^{\mathfrak{d}_0+\mathfrak{d}_1}-1}{p-1}.
\end{align}  Because it will be particularly useful in a moment, let us also note that
\begin{align}\label{eq:count.on.Hp3}
\nu(H_{p^3} \twoheadrightarrow \Z/p,K/F) = p^{\mathfrak{d}_0+\mathfrak{d}_1-1}\frac{p^{\mathfrak{d}_1}-1}{p-1}.
\end{align}

One can solve for $p^{\mathfrak{d}_0+\mathfrak{d}_1}$ and $p^{\mathfrak{d}_1}$ using equations (\ref{eq:count.on.Zp.times.Zp}) and (\ref{eq:count.on.Hp3}), and ultimately recover 
\begin{align*}
p^{\mathfrak{d}_0+\mathfrak{d}_1} &= 1 + (p-1)\nu(\Z/p\times\Z/p \twoheadrightarrow \Z/p,K/F)\\
p^{\mathfrak{d}_1} &= 1+\frac{p(p-1)}{p^{\mathfrak{d}_0+\mathfrak{d}_1}}\nu(H_{p^3}\twoheadrightarrow\Z/p,K/F)\\
&=1+ \frac{p(p-1)\nu(H_{p^3}\twoheadrightarrow \Z/p,K/F)}{1+(p-1)\nu(\Z/p\times\Z/p \twoheadrightarrow\Z/p,K/F)}
\end{align*}

With these in hand, we can reexpress (\ref{eq:realization.count.rewritten}) to satisfy the statement of the Proposition:
\begin{align*}\nu(A_\ell \rtimes \Z/p &\twoheadrightarrow\Z/p,K/F) =
p^{\mathfrak{d}_0+\mathfrak{d}_1-1}\frac{p^{\mathfrak{d}_1}-1}{p-1}\left(p^{\mathfrak{d}_1-1}\right)^{\ell-2} \\
&= \nu(H_{p^3} \twoheadrightarrow \Z/p,K/F) \left(\frac{1}{p}+\frac{(p-1)\nu(H_{p^3}\twoheadrightarrow \Z/p,K/F)}{1+(p-1)\nu(\Z/p\times\Z/p\twoheadrightarrow\Z/p,K/F)}\right)^{\ell-2}.
\end{align*}
\end{proof}

We finish with some realization multiplicity results for groups from the family $A_i \rtimes G_1$ and $A_i \bullet G_1$.  The realization multiplicity of $G$, written $\nu(G)$, is the minimal positive value for $\nu(G,F)$ as $F$ ranges over all fields; likewise, the realization multiplicity for an embedding problem $G \twoheadrightarrow Q$, written $\nu(G \twoheadrightarrow Q)$, is the minimal positive value for $\nu(G \twoheadrightarrow Q, K/F)$ as $K/F$ ranges over all fields with $\Gal(K/F) \simeq Q$.  There are very few realization multiplicity results known for nonabelian $p$-groups; the known results come from \cite{J2} and are $\nu(M_{p^3}) = p$, $\nu(M_{p^3}\times \Z/p) = p^2-1$, and $\nu((\Z/p)^k \times H_{p^3}) = 1$ for $k \in \Z_{\geq 0}$.



\begin{example}\label{ex:unique.hp3}
Consider a field $F$ with $\textrm{char}(F) \neq p$ and $\dim_{\F_p}(J(F)) = 2$ and such that $G_F:=\Gal(F_{\textrm{sep}}/F)$ is the free pro-$p$ group on two generators.  (Such a field exists as seen in \cite[Cor.~23.1.2]{FJ}.)  Now we claim that $\xi_p \in F$, since otherwise we would have $F(\xi_p) \subseteq F_{\textrm{sep}}$ and $1<[F(\xi_p):F] \leq p-1 <p$.  But this would imply that $G_F$ has a quotient whose order was not a power of $p$, a clear contradiction.  

Since $G_F$ is a free pro-$p$ group we have $H^2(F) = 0$, and therefore $(f_1) \cup (f_2) = 0$ for each $f_1,f_2 \in J(F)$.  But then we conclude that $f_1 \in N_{F(\root{p}\of{f_2})/F}(F(\root{p}\of{f_2}))$ for all $f_1,f_2 \in F^\times$.  Suppose, then, that $f_1,f_2$ are generators for $J(F)$.  If we let $\alpha \in F(\root{p}\of{f_2})$ be given so that $N_{F(\root{p}\of{f_2})/F}(\alpha) = f_1$, and if we write $\sigma$ for the generator of the $G_1$-extension $F(\root{p}\of{f_2})/F$, then $\hat \alpha = \alpha^{(\sigma-1)^{p-2}}$ has $\ell(\hat \alpha) = 2$ and $e(\hat \alpha) = 0$.  Hence $\langle \hat \alpha \rangle$ corresponds to a solution to the embedding problem $H_{p^3} \twoheadrightarrow \Z/p$.


\end{example}

\begin{corollary}\ 
\begin{enumerate}
\item\label{it:real.mult.hp3} $\nu(H_{p^3}) = 1$;
\item\label{it:real.mult.g1.by.ap} $\nu(A_p \rtimes \Z/p) = p^2-1$
\item\label{it:real.mult.g1.split.ai} $\nu(A_i \rtimes \Z/p) =p+1$ for $2 < i < p-1$
\item\label{it:real.mult.g1.nonsplit.ai} $\nu(A_i \bullet \Z/p) = p^2-1$ for $2<i<p-1$; and
\end{enumerate}
\end{corollary}

\begin{remark*}
Though it is already known, we reprove statement (\ref{it:real.mult.hp3}) for two reasons: to again showcase a module-theoretic perspective, and because statement (\ref{it:real.mult.hp3}) provides the necessary example in proving statements (\ref{it:real.mult.g1.by.ap}-\ref{it:real.mult.g1.nonsplit.ai}).
\end{remark*}

\begin{proof}
Statement (\ref{it:real.mult.hp3}) is a consequence of Example \ref{ex:unique.hp3}. In that example we are told that there is a $\Z/p$-extension $K/F$ for which the embedding problem $H_{p^3} \twoheadrightarrow \Z/p$ is solvable; call this $H_{p^3}$-extension $L/F$.  Combined with the fact that $\ker|J(F) \to J(K)| = p$, one can use this to show that  $J(K) \simeq X \oplus \F_p[G_1]$, with $\dim_{\F_p}(X) = 1$.  But note that in fact $L/F$ is solution to the embedding problem $H_{p^3} \twoheadrightarrow \Z/p$ for \emph{any} $\Z/p$-extension $\tilde K/F$, from which we can deduce that $J(\tilde K) \simeq J(K)$.  Hence over each $\Z/p$-extension of $F$ there is a unique module isomorphic to $A_2$ and generated by an element of trivial index.  Each of these modules corresponds to the same $H_{p^3}$-extension $L/F$.  

One can also prove this result by thinking of $H_{p^3}$ as an extension of $\Z/p \times \Z/p$ by $\Z/p$.  Consider the field $F$ from Example \ref{ex:unique.hp3} again, and note that $F$ has a unique $\Z/p \times \Z/p$-extension $K$.  Because any $H_{p^3}$-extension of $F$ contains a unique $\Z/p \times \Z/p$ quotient extension over $F$, we see that any $H_{p^3}$-extension of $F$ contains $K/F$.  Now recall that \cite[Thm.~6.6.1]{JLY} tells us that if $K = F(\root{p}\of{a},\root{p}\of{b})$, and if $w \in F(\root{p}\of{a})$ is an element so that $L = F(\root{p}\of{w},\root{p}\of{b})$ is an $H_{p^3}$-extension of $F$, then all other solutions to the $H_{p^3}$ embedding problem over $K/F$ take the form $L_f := K(\root{p}\of{fw},\root{p}\of{b})$.  But notice that since $K/F$ is the unique $\Z/p \times \Z/p$-extension of $F$, we have $\root{p}\of{f} \in K$; it follows that $L_f = L$, and so $F$ has a unique $H_{p^3}$-extension.


The proofs of statements (\ref{it:real.mult.g1.by.ap}--\ref{it:real.mult.g1.nonsplit.ai}) are all relatively similar and use Example \ref{ex:unique.hp3} to establish upper bounds; we will prove (\ref{it:real.mult.g1.by.ap}) and leave the verification of the other two statements to the reader.  So let $F$ be the field from Example \ref{ex:unique.hp3}, and note that for each $\Z/p$-extension $K/F$ there are $p$ modules isomorphic to $A_p$.  Each of these corresponds to a distinct solution to the embedding problem $A_p \rtimes G_1 \twoheadrightarrow G_1$.  By Lemma \ref{le:different.extensions.for.long.modules}, the solutions to this embedding problem over each of the $\Z/p$-extensions of $F$ are distinct.  Hence we have $\nu(A_p \rtimes G_1) \leq p^2+p$.  

For the lower bound, observe that if there is a single $A_p \rtimes G_1$-extension of a field $F$, then the $\Z/p$-subextension $K/F$ corresponding to the natural projection $A_p \rtimes G_1 \twoheadrightarrow G_1$ has the property that $J(K)$ contains a module $\langle \gamma \rangle$ isomorphic to $A_p$.  If we choose an element $\chi$ as in Proposition \ref{prop:decomposition}, then each of the modules $\langle \gamma +c\chi \rangle$ for $c \in \{0,1,\cdots,p-1\}$ are also isomorphic to $A_p$.  Observe additionally that any of the $p$ other $\Z/p$-extensions of $F$ admit a solution to the embedding problem $A_2 \rtimes G_1$ (i.e., the $H_{p^3}$ subextension from the original $A_p \rtimes G_1$ extension), and hence by Proposition \ref{prop:automatic.realizations}(\ref{it:split.realizes.split}) also admits at least one solution to the embedding problem $A_p \rtimes G_1 \twoheadrightarrow G_1$.  Since we have already argued that one such solution forces the appearance of $p$ solutions, this tells us that $F$ must have at least $p^2+p$ extensions with Galois group $A_p \rtimes G_1$.
\end{proof}


\begin{thebibliography}{99}

\bibitem{AnFu} {\sc F.~Anderson}, {\sc K.~Fuller}. \emph{Rings and categories of modules}. Graduate Texts in Mathematics 13. New
York: Springer-Verlag, 1973.


\bibitem{BP} {\sc F.~Bertrandias}, {\sc J.J.~Payan}. $\Gamma $-extensions et invariants cyclotomiques. \emph{Ann.Scient.Ec.Norm.Sup.}, {\bf 5} (1972), no.~4, 517--543 .

\bibitem{BT1} {\sc F.~Bogomolov}, {\sc Y.~Tschinkel}. Universal spaces for unramified Galois cohomology. Available at \url{http://www.math.nyu.edu/~tschinke/papers/yuri/14bloch/bloch18.pdf}.  To appear in Brauer groups and obstructions problems: moduli spaces and arithmetic, (A.~Auel, B.~Hassett, A.~Varilly-Alvarado, and B.~Viray eds.), (2014).

\bibitem{BT2} {\sc F.~Bogomolov}, {\sc Y.~Tschinkel}. Commuting elements in Galois groups of function fields. In \emph{Motives, polylogarithms and Hodge theory}, edited by F.~Bogomolov, L.~Katzarkov.  International Press (2002), 75--120.

\bibitem{BT3} {\sc F.~Bogomolov}, {\sc Y.~Tschinkel}. Introduction to birational anabelian geometry.  In \emph{Current developments in algebraic geometry}, edited by L.~Caporaso, J.~McKernan, M.~Mustata, M.~Popa, volume 59 of MSRI Publications,  Cambridge University Press (2012),  17--63.

\bibitem{Br} {\sc G.~Brattstr\"{o}m}. On p-groups as Galois groups. \textit{Math.~Scandinavica} {\bf 65} (1989), no.~2, 165--174.

\bibitem{CEM} {\sc S.~Chebolu}, {\sc I.~Efrat}, {\sc J.~Min\'a\v{c}}. Quotients of absolute Galois groups which determine the entire Galois cohomology. \emph{Math.~Annalen.} {\bf 352} (2012), no.~1, 205--221.

\bibitem{DF} {\sc D.~Dummit}, {\sc R.~Foote}. \emph{Abstract algebra, 2nd ed.} Upper Saddle River, NJ: Prentice Hall, 1999.

\bibitem{Ef} {\sc I.~Efrat}. The Zassenhaus filtration, Massey products, and representations of profinite groups. \emph{Adv.~Math} {\bf 263} (2014), 389--411.

\bibitem{Ef2} {\sc I.~Efrat}. Filtrations of free groups as intersections.  \emph{Arch.~Math.} (Basel) {\bf 103} (2014), no.~5, 411--420.  Available at \url{http://arxiv.org/pdf/1312.1811}.

\bibitem{EM11} {\sc I.~Efrat}, {\sc J.~Min\' a\v c}. On the descending central sequence of absolute Galois groups. \emph{Amer.~J.~Math.}, {\bf 133} (2011), 1503--1532.

\bibitem{EM2} {\sc I.~Efrat}, {\sc J.~Min\' a\v c}. Galois groups and cohomological functors. \emph{Trans.~Amer.~Math.~Soc.} (2016), \url{http://dx.doi.org/10.1090/tran/6724}, in press.  Available at \url{http://arxiv.org/abs/1103.1508}.

\bibitem{EM12} {\sc I.~Efrat}, {\sc J.~Min\' a\v c}. Small Galois groups that encode valuations. \emph{Acta.~Arith.} {\bf 156} (2012), no.~1, 7--17.

\bibitem{FJ} {\sc M.~Fried}, {\sc M.~Jarden}. \emph{Field Arithmetic}.  Ergebnisse der Mathematik und ihrer Grenzgebiete, Vol.~11.  Berlin: Springer Berlin Heidelberg, 2005.  

\bibitem{GS} {\sc P.~Gille}, {\sc T.~Szamuely}. \emph{Central simple algebras and Galois cohomology.} Cambridge studies in advanced mathematics 101, Cambridge University Press, 2006.

\bibitem{Hall} {\sc M.~Hall}. \textit{The theory of groups}. Macmillian Company, New York, 1959.

\bibitem{HW} {\sc M.~Hopkins}, {\sc K.~Wickelgren}. Splitting varieties for triple Massey products.  \emph{J.~Pure Appl. Algebra}, {\bf 219} (2015), no.~5, 1304--1319.  Available at \url{http://arxiv.org/pdf/1210.4964}.

\bibitem{ILF} {\sc V.V.~Ishkhanov}, {\sc B.B.~Lur$'$e}, {\sc D.K.~Faddeev}. \emph{The embedding problem in Galois theory}. Translations in mathematical monographs, vol.~165. Amer.~Math.~Soc., Providence, 1997.

\bibitem{J1} {\sc C.~U.~Jensen}. On the representations of a group as a Galois group over an arbitrary field. Th\'eorie des nombres (Quebec, PQ, 1987), 441--458. Berlin: de Gruyter, 1989.

\bibitem{J2} {\sc C.U.~Jensen}. Finite groups as Galois groups over arbitrary fields.  Proceedings of the International Conference on Algebra, Part 2 (Novosibirsk, 1989), 435--448. Contemp.~Math.  {\bf 131}, Part 2. Providence, RI: American Mathematical Society, 1992.


\bibitem{JLY} {\sc C.U.~Jensen}, {\sc A.~Ledet}, {\sc N.~Yui}. \emph{Generic polynomials: constructive aspects of the inverse Galois problem}. Mathematical Sciences Research Institute Publications 45. Cambridge: Cambridge University Press, 2002.


\bibitem{KC} {\sc V.~Kac}, {\sc P.~Cheung}. Quantum calculus. Springer, 2002.

\bibitem{Lang} {\sc S.~Lang}. \emph{Algebra}, 3rd ed. Graduate texts in mathematics, 211.  Springer, 2005.

\bibitem{Lbook} {\sc A.~Ledet}. Brauer type embedding problems. Fields Institute Monographs 21. Providence, RI: AMS, 2005.

\bibitem{LMSSembed}  {\sc N.~Lemire}, {\sc J.~Min\'a\v{c}}, {\sc A.~Schultz}, {\sc J.~Swallow}. Galois module structure of Galois cohomology for embeddable cyclic extensions of degree $p^n$. \emph{J.~London Math.~Soc.} {\bf 81} (2010), no.~3, 525--543.

\bibitem{Mass} {\sc R.~Massy}. Construction de $p$-extensions Galoisiennes d'un corps de caract\'{e}ristique diff\'{e}rente de $p$. \emph{J.~Algebra} {\bf 109} (1987), 508--535.


\bibitem{Mich1} {\sc I.~Michailov}. Induced orthogonal representations of Galois groups. \textit{J.~Alg.} {\bf 322} (2009), 3713--3732.

\bibitem{Mich2} {\sc I.~Michailov}. Four non-abelian groups of order $p^4$ as Galois groups. \textit{J.~Alg.} {\bf 307} (2007), 287--299.

\bibitem{Mi} {\sc J.~Milnor}. Algebraic $K$-theory and quadratic forms. \emph{Inventiones Math.} {\bf 9} (1970), 318--344.

\bibitem{MT4} {\sc J.~Min\'a\v{c}}, {\sc M.~Rogelstad}, {\sc N.D.~Tan}. Dimensions of Zassenhaus filtration subquotients of some pro-$p$-groups. \emph{Israel J.~Math.} {\bf 212} (2016), no.~2, 825--855.   Available at \url{http://arxiv.org/pdf/1405.6980}.

\bibitem{MSS1} {\sc J.~Min\'a\v{c}}, {\sc A.~Schultz}, {\sc J.~Swallow}. Galois module structure of the $p$th-power classes of cyclic extensions of degree $p^n$.  Proc.~London Math.~Soc. {\bf 92} (2006), no.~2, 307--341.

\bibitem{MSSauto} {\sc J.~Min\'a\v{c}}, {\sc A.~Schultz}, {\sc J.~Swallow}. Automatic realizations of Galois groups with cyclic quotient of order $p^n$. \emph{J.~Th\'{e}or.~Nombres Bordeaux} {\bf 20} (2008), 419--430.

\bibitem{MSp} {\sc J.~Min\'a\v{c}}, {\sc M.~Spira}. Witt rings and Galois groups.  \emph{Ann.~Math.} {\bf 144} (1996), 36--60.

\bibitem{MSp2} {\sc J.~Min\'a\v{c}}, {\sc M.~Spira}. Formally real fields, Pythagorean fields, $C$-fields and $W$-groups. \emph{Math.~Z.} {\bf 205} (1990), 519--530.

\bibitem{MSauto} {\sc J.~Min\'a\v{c}}, {\sc J.~Swallow}.  Galois embedding problems with cyclic quotient of order $p$. Israel J.~Math.
{\bf 145} (2005), 93--112.


\bibitem{MT1} {\sc J.~Min\'a\v{c}}, {\sc N.D.~Tan}. Triple Massey products over global fields. \emph{Doc.~Math.} {\bf 20} (2015), 1467--1480.  Available at \url{http://arxiv.org/abs/1403.4586}. 

\bibitem{MT2} {\sc J.~Min\'a\v{c}}, {\sc N.D.~Tan}. Triple Massey products and Galois theory. \emph{J.~Eur.~Math.~Soc.} (JEMS) (2016), in press. Available at \url{http://arxiv.org/abs/1307.6624}. 

\bibitem{MT3} {\sc J.~Min\'a\v{c}}, {\sc N.D.~Tan}. The kernel unipotent conjecture and the vanishing of Massey products of odd rigid fields. (With an appendix by {\sc I.~Efrat}, {\sc J.~Min\'{a}\v{c}}, and {\sc N.D.~Tan}.) \emph{Adv.~Math.} {\bf 273} (2015), 242--270. Available at \url{http://arxiv.org/abs/1312.2655}.


\bibitem{NSW} {\sc J.~Neukirch}, {\sc A.~Schmidt}, {\sc K.~Wingberg}. \emph{Cohomology of Number Fields}, 2nd ed. Berlin: Springer-Verlag, 2008.

\bibitem{Pf} {\sc A.~Pfister}. Eine Bemerkung zum normenresthomomorphismus $h:K^*F/2 \to H^*(F,\Z/2)$. \emph{Arch.~Math.} {\bf 81} (2003), 272--284.

\bibitem{P2} {\sc A.~Pfister}. On the Milnor conjectures: history, influence, applications. \emph{Jber.~d.~Dt.~Math-Verein.} {\bf 102} (2000), 15--41.

\bibitem{Pierce} {\sc R.C.~Pierce}. Associative algebras. Graduate Texts in Mathematics, Vol.~88.  Springer-Verlag, 1982.

\bibitem{Pop1} F.~Pop.  On Grothendieck's conjecture of birational anabelian geometry. \emph{Ann.~of Math.} {\bf 139} (1994), 145--182.

\bibitem{Pop2} F.~Pop.  On the birational program initiated by Bogomolov I.  \emph{Invent.~Math.} {\bf 187} (2012), 511--533.

\bibitem{Schultz2} {\sc A.~Schultz}. Parameterizing solutions to any Galois embedding problem over $\Z/p^n\Z$ with elementary $p$-abelian kernel.  \emph{J.~Algebra}. {\bf 411} (2014), 50--91.

\bibitem{Sha1} {\sc I.R.~Shafarevich}. Construction of fields of algebraic number fields with given solvable groups (in Russian).  \emph{Izv.~Akad.~Nauk SSSR} {\bf 18} (1954), no.~6, 525--578.  English translation in \emph{Amer.~Math.~Soc.~Transl.} {\bf 4} (1956), 185--237.

\bibitem{Sha2} {\sc I.R.~Shafarevich}. Factors of a decreasing central series (in Russian). \emph{Mat.~Zametki} {\bf 45} (1989), 114--117.  English transition in \emph{Math.~Notes} {\bf 45} (1989), 262--264.

\bibitem{Sha3} {\sc I.R.~Shafarevich}. On the construction of fields with a given Galois group of order $\ell^\alpha$ (in Russian). \emph{Izv.~Akad.~Nauk SSSR} {\bf 18} (1954), 261--296.  English transition in \emph{Amer.~Mat.~Soc.~Transl.} {\bf 4} (1956), 107--142.

\bibitem{Srin} {\sc V.~Srinivas}. \emph{Algebraic $K$-theory,} reprint of 2nd ed.  Part of \emph{Modern Birkh\"{a}user classics}.  Boston, MA: Birkh\"{a}user, 2008.


\bibitem{T1} {\sc A.~Topaz}. Abelian-by-central Galois groups of fields I: a formal description. To appear in \emph{Trans.~Amer.~Math.~Soc.} Available at \url{http://arxiv.org/pdf/1310.5613}.

\bibitem{T2} {\sc A.Topaz}. Commuting-liftable subgroups of Galois groups II. To appear in \emph{J.~Reine Angew.~Math.}  Available at \url{http://arxiv.org/pdf/1208.0583}.

\bibitem{V}{\sc V. Voevodsky}. On motivic cohomology with $\mathbf{Z}/l$-coefficients. {\bf 174} (2011), 401--438.

\bibitem{Wat} {\sc W.~Waterhouse}.  The normal closures of certain Kummer extensions. Canad. Math.~Bull. {\bf 37} (1994), no.~1,
133--139.

\bibitem{Wh} {\sc G.~Whaples}. Algebraic extensions of arbitrary fields. \emph{Duke Math.~J.} {\bf 24} (1957), 201--204.

\bibitem{Wi} {\sc E.~Witt}. Konstruktion von galoisschen k\"{o}rpern der charakteristik $p$ zu vorgegebener gruppe der ordnung $p^f$. \emph{J.~Reine Angew.~Math.} {\bf 174} (1936), 237--245.


\end{thebibliography}
\end{document}